\documentclass[11pt, reqno]{amsart}
\usepackage{hyperref}
\usepackage{graphicx}
\usepackage[shortlabels]{enumitem}
\usepackage[ansinew]{inputenc} 

\usepackage{amssymb}
\usepackage{bbm}

\usepackage{mathtools}
\usepackage{amsthm}
\usepackage{cases}
\usepackage{MnSymbol}

\usepackage{enumitem}
\usepackage{fullpage}
\usepackage{pifont} 

\renewcommand{\Re}{\operatorname{Re}}
\renewcommand{\Im}{\operatorname{Im}}

\makeatletter
\def\@tocline#1#2#3#4#5#6#7{\relax
  \ifnum #1>\c@tocdepth 
  \else
    \par \addpenalty\@secpenalty\addvspace{#2}%
    \begingroup \hyphenpenalty\@M
    \@ifempty{#4}{%
      \@tempdima\csname r@tocindent\number#1\endcsname\relax
    }{%
      \@tempdima#4\relax
    }%
    \parindent\z@ \leftskip#3\relax \advance\leftskip\@tempdima\relax
    \rightskip\@pnumwidth plus4em \parfillskip-\@pnumwidth
    #5\leavevmode\hskip-\@tempdima
      \ifcase #1
       \or\or \hskip 1em \or \hskip 2em \else \hskip 3em \fi%
      #6\nobreak\relax
    \dotfill\hbox to\@pnumwidth{\@tocpagenum{#7}}\par
    \nobreak
    \endgroup
  \fi}
\makeatother

\newcommand{\bD}{\mathbb{D}}

\newcommand{\E}{\mathbb E }
\newcommand{\R}{\mathbb{R}}
\newcommand{\N}{\mathbb{N}}
\newcommand{\C}{\mathbb{C}}
\newcommand{\Z}{\mathbb{Z}}

\newcommand{\ind}{{\mathbbm{1}}}

\newcommand{\supp}{\mathop{\mathrm{supp}}\nolimits}

\newcommand{\eee}{{\rm e}}
\newcommand{\ii}{{\rm{i}}}

\newcommand{\toweak}{\overset{w}{\underset{n\to\infty}\longrightarrow}}

\newcommand{\ton}{\overset{}{\underset{n\to\infty}\longrightarrow}}

\newcommand{\sgn}{\mathop{\mathrm{sgn}}\nolimits}

\newcommand{\unif}{{\rm Unif}}
\newcommand{\bsl}{\backslash}

\newcommand{\dd}{{\rm d}}

\newcommand{\stirling}[2]{\genfrac{[}{]}{0pt}{}{#1}{#2}}
\newcommand{\stirlingsec}[2]{\genfrac{\{}{\}}{0pt}{}{#1}{#2}}
\newcommand{\eulerian}[2]{\genfrac{\langle}{\rangle}{0pt}{}{#1}{#2}}

\newcommand{\HyperP}[3]{\mathcal{H}^{#1,#2}_{#3}}

\newcommand{\eps}{\varepsilon}

\theoremstyle{plain}
\newtheorem{theorem}{Theorem}[section]
\newtheorem{proposition}[theorem]{Proposition}

\newtheorem{lemma}[theorem]{Lemma}
\newtheorem{corollary}[theorem]{Corollary}

\theoremstyle{definition}
\newtheorem{definition}[theorem]{Definition}

\theoremstyle{remark}
\newtheorem{remark}[theorem]{Remark}

\makeatletter
\@namedef{subjclassname@2020}{\textup{2020} Mathematics Subject Classification}
\makeatother

\begin{document}

\title{Zeros and exponential profiles of polynomials II:
\\
Examples}

\begin{abstract}
In [Jalowy, Kabluchko, Marynych, \href{https://arxiv.org/pdf/2504.11593}{arXiv:2504.11593v1}, 2025], the authors discuss a user-friendly approach to determine the limiting empirical zero distribution of a sequence of real-rooted polynomials, as the degree goes to $\infty$. In this note, we aim to apply it to a vast range of examples of polynomials providing a unifying source for limiting empirical zero distributions.

We cover Touchard, Fubini, Eulerian, Narayana  and little $q$-Laguerre polynomials as well as hypergeometric polynomials including the classical Hermite, Laguerre and Jacobi polynomials. We construct polynomials whose empirical zero distributions converge to the free multiplicative normal and Poisson distributions.  Furthermore, we study polynomials generated by some differential operators.
As one inverse result, we derive coefficient asymptotics of the characteristic polynomial of random covariance matrices.
\end{abstract}

\author{Jonas Jalowy}
\address{Jonas Jalowy: Paderborn University, Institute of Mathematics, Warburger Str. 100, 33098 Paderborn, Germany}
\email{jjalowy@math.upb.de}

\author{Zakhar Kabluchko}

\address{Zakhar Kabluchko: Institute of Mathematical Stochastics, Department of Mathematics and Computer Science, University of M\"{u}nster, Orl\'{e}ans-Ring 10, D-48149 M\"{u}nster, Germany}
\email{zakhar.kabluchko@uni-muenster.de}

\author{Alexander Marynych}
\address{Alexander Marynych: School of Mathematical Sciences, Queen Mary University of London, Mile End Road, London E14NS, United Kingdom; Faculty of Computer Science and Cybernetics, Taras Shevchenko National University of Kyiv, Kyiv 01601, Ukraine}
\email{o.marynych@qmul.ac.uk, marynych@knu.ua}


\maketitle

\tableofcontents

\section{Introduction}
We consider a sequence of polynomials $P_n(x) = \sum_{k=0}^n a_{k:n} x^k$ of degree at most $n$, with real coefficients and nonpositive real zeros. This paper is a continuation of the recent study \cite{jalowy_kabluchko_marynych_zeros_profiles_part_I}, where the authors discuss the method of exponential profiles for establishing the existence of a limiting empirical zero distribution of $P_n$ as $n\to\infty$. While \cite{jalowy_kabluchko_marynych_zeros_profiles_part_I} focused on the theoretical foundations and examined how limiting empirical zero distributions are affected by certain polynomial operations, such as finite free convolution and repeated differentiation, the present paper demonstrates the versatility of the method by applying it to various classical polynomials. Although we shall briefly recall the necessary definitions and tools below, we recommend that readers first familiarize themselves with \cite{jalowy_kabluchko_marynych_zeros_profiles_part_I}, where they will find detailed bibliographic remarks, motivation, and further applications.

\section{A quick recap on zeros and exponential profiles}

\subsection{Definitions}

We shall briefly recall necessary definitions and concepts from \cite{jalowy_kabluchko_marynych_zeros_profiles_part_I} and refer to that source for detailed explanations.

\begin{definition}\label{def:empirical_distr}
The \textit{empirical distribution of zeros} of a polynomial $P\not\equiv 0$ of degree at most $n$ is the probability measure $\lsem P \rsem_n$ on $\overline{\C}:=\C\cup\{\infty\}$ which assigns equal weight $1/n$ to each zero (counting multiplicities) and places the remaining mass at $\infty$. Thus, letting $\delta_x$ denote the Dirac measure at $x\in\overline{\mathbb{C}}$, we define
\begin{equation}\label{eq:empirical_distr_zeros_def}
\lsem P \rsem_n := \frac 1n \sum_{\substack{z\in \C\,:\, P(z) = 0}} \delta_z+ \frac{n-\deg P}{n}\delta_{\infty}.
\end{equation}
\end{definition}
We shall work with sequences of polynomials rather than with a single polynomial $P$, and the variable $n$ in Definition~\ref{def:empirical_distr} should be understood as the variable indexing some sequence $(P_n)_{n \in \N}$. In many cases, $n$ coincides with the degree of $P_n$, but in general, we shall assume only $\deg P_n \leq n$. If a polynomial $P$ has only real zeros all of the same sign, we identify the point at infinity, $\infty$, with either $+\infty$ (when the roots are nonnegative) or $-\infty$ (when the roots are nonpositive). Accordingly, $\lsem P\rsem_n$ is supported on $[0,+\infty]$ in the former case, and on $[-\infty,0]$ in the latter.

An important characteristic of a sequence of polynomials $(P_n)$ with {\em nonnegative real coefficients} is its exponential profile.

\begin{definition}\label{def:exp_profile}
For every $n\in \N$ let $P_n(x) = \sum_{k=0}^n a_{k:n} x^k$ be a polynomial of degree at most $n$ with nonnegative coefficients $a_{k:n}\ge 0$.  We say that the sequence $(P_n)_{n\in \N}$ has an \emph{exponential profile} $g$  if $g:(\underline{m},\overline{m})\to \R$ is a function defined on a nonempty interval $(\underline{m},\overline{m})\subseteq [0,1]$  such that
\begin{equation}\label{eq:exp_profile_definition}
\lim_{n\to\infty}\frac{1}{n}\log a_{\lfloor \alpha n \rfloor:n}=
\begin{cases}
g(\alpha),&\text{if }\alpha\in (\underline{m},\overline{m}),\\
-\infty,&\text{if }\alpha\in [0,\underline{m})\cup (\overline{m},1].
\end{cases}
\end{equation}
Here, and in the following, we stipulate that $\log 0 = -\infty$.
\end{definition}

The basic principle lying behind the method of exponential profiles discussed in \cite{jalowy_kabluchko_marynych_zeros_profiles_part_I} is as follows.

\begin{center}
\fbox{\begin{minipage}{0.8\textwidth}
A sequence $(P_n)_{n\in\mathbb{N}}$ of polynomials with only nonpositive roots possesses an exponential profile if and only if the sequence of probability measures $(\lsem P_n\rsem_n)_{n\in\mathbb{N}}$ converges weakly on $[-\infty,\,0]$ to a limiting probability measure which assigns a positive mass to $(-\infty,0)$.
\end{minipage}}
\end{center}

Theorems~\ref{theo:exp_profile_implies_zeros} and~\ref{theo:zeros_imply_exp_profile}, established in \cite{jalowy_kabluchko_marynych_zeros_profiles_part_I} and restated below for the reader's convenience, provide a rigorous foundation for the aforementioned principle. For the purposes of the present paper, the direct implication stating that existence of an exponential profile implies convergence of the empirical distribution is of greater relevance. Indeed, given a sequence of polynomials with only nonpositive roots, a direct proof of the convergence of their empirical distributions is typically a technically involved task. By contrast, establishing the existence of an exponential profile usually follows straightforwardly from classical asymptotic relations.

The main principle outlined above suggests that it should be possible to recover the limiting measure of $\lsem P_n \rsem_n$ as $n \to \infty$ from the exponential profile $g$, and conversely. The most convenient framework for this purpose is provided by Cauchy transforms.

\begin{definition}\label{}
For any probability measure $\mu$ on $\R$ its \emph{Cauchy transform} $G:\C\backslash \supp(\mu)\to\C$ is defined by
\begin{align}
G(t)=\int_{\R}\frac{\mu({\rm d}z)}{t-z}.
\end{align}
\end{definition}

The negative of the Cauchy transform is known as the Stieltjes transform. By the Stieltjes--Perron inversion formula, see~\cite[Proposition 2.1.2 on p.~35]{pastur_shcherbina_book} and~\cite[pp.~124--125]{akhiezer_book}, the Cauchy transform uniquely determines the measure $\mu$ via
\begin{equation}\label{eq:Perron_inversion}
\mu (I) =
-\lim_{y\to 0+} \int_I\frac 1 \pi\Im G(x + iy){\rm d}x,
\end{equation}
where $I$ is any interval such that $\mu$ is continuous at its endpoints. Also, locally uniform convergence on $\mathbb{C}\backslash\R$ of the sequence of Cauchy transforms is equivalent to weak convergence of the corresponding sequence of probability measures.

Observe that the definition of Cauchy transforms extends naturally to probability measures on $(-\infty,+\infty]$ or $[-\infty,+\infty)$. Moreover, if $\mu$ is supported by $[-\infty,A]$ (respectively, $[A,+\infty]$) for some $A\in\mathbb{R}$, then
\begin{equation}\label{eq:atom_at_infinity_over_R}
1 -\mu(\{-\infty\}) = \lim_{t\to+\infty} t G(t)\quad\quad \left(\text{respectively, }1 -\mu(\{+\infty\}) = \lim_{t\to-\infty} t G(t)\right).
\end{equation}

\subsection{Theorems on zeros and profiles} Having recalled the necessary concepts, we are now ready to formulate two theorems, proved in~\cite{jalowy_kabluchko_marynych_zeros_profiles_part_I}, that support the method of exponential profiles.

\begin{theorem}[Profile for coefficients implies limiting distribution of zeros, Theorem 2.2 in \cite{jalowy_kabluchko_marynych_zeros_profiles_part_I}]
\label{theo:exp_profile_implies_zeros}
For every $n\in \N$ let $P_n(x) = \sum_{k=0}^n a_{k:n} x^k$ be a polynomial of degree at most $n$. Suppose that the sequence $(P_n)_{n\in\N}$ possesses an exponential profile $g:(\underline{m},\overline{m})\to\mathbb{R}$ in the sense of Definition~\ref{def:exp_profile}. If, additionally, all roots of $P_n$ are nonpositive for all sufficiently large $n$,  then the following hold.
\begin{enumerate}[(a)]
\item We have weak convergence
$$\lsem P_n\rsem_n\toweak\mu$$
of probability measures on $[-\infty,0]$ for some $\mu$.
\item The measure $\mu$ is uniquely determined by the profile $g$ as follows. If $G$ denotes the Cauchy transform of $\mu$, then $t\mapsto tG(t)$, $t>0$, is the inverse of the function\footnote{The derivative $g'$ exists on $(\underline{m},\overline{m})$. In fact, $g$ is automatically infinitely differentiable and strictly concave on $(\underline m ,\overline m )$.} $\alpha\mapsto\eee^{-g'(\alpha)}$, $\alpha\in (\underline{m},\overline{m})$. The function $t\mapsto tG(t)$ is a strictly increasing continuous bijection between $(0,+\infty)$ and $(\underline{m},\overline{m})$.
\item It holds $\mu(\{0\})=\underline{m}$ and $\mu(\{-\infty\})=1-\overline{m}$.
\end{enumerate}
\end{theorem}

The converse of Theorem~\ref{theo:exp_profile_implies_zeros} is as follows.

\begin{theorem}[Distribution of zeros implies exponential profile, Theorem 2.4 in \cite{jalowy_kabluchko_marynych_zeros_profiles_part_I}]
\label{theo:zeros_imply_exp_profile}
For every $n\in \N$ let $P_n(x) = \sum_{k=0}^n a_{k:n} x^k$ be a polynomial of degree at most $n$ with real, nonpositive roots. If $\lsem P_n\rsem_n$ converges weakly to some probability measure $\mu$ on $[-\infty, 0]$, then the following hold.
\begin{enumerate}[(a)]
\item Define $\underline{m}:= \mu(\{0\})$ and $\overline{m} := 1-\mu(\{-\infty\})$ and observe that $0\leq \underline{m} \leq \overline{m}\leq 1$.  There exists a strictly concave and infinitely differentiable function $g:(\underline{m},\overline{m})\to (-\infty,0]$ such that
\begin{equation}\label{eq:zeros_imply_exp_profile_claim}
\sup_{(\underline{m}+\eps) n \leq k  \leq (\overline{m}-\eps)n}
\left|\frac 1n  \log \frac {a_{k :n}}{P_n(1)} -g\left(\frac k  n\right)  \right| \ton 0,
\end{equation}
for all sufficiently small $\eps > 0$ as well as
\begin{equation}\label{eq:divergence_outside_polys}
\sup_{0\leq k  \leq (\underline{m}-\eps)n} \frac 1n   \log \frac {a_{k :n}}{P_n(1)} \ton  -\infty,
\qquad
\sup_{(\overline{m}+\eps)n \leq k  \leq n} \frac 1n   \log \frac {a_{k :n}}{P_n(1)} \ton  -\infty.
\end{equation}
\item If $\underline{m}  < \overline{m}$, then the Cauchy transform $G$ of $\mu$ is such that $(0,+\infty)\ni t\mapsto tG(t)$ is the inverse of $(\underline{m},\overline{m})\ni\alpha\mapsto \eee^{-g'(\alpha)}$.
\item If $\underline{m}=\overline{m}$, then~\eqref{eq:zeros_imply_exp_profile_claim} is a void statement but~\eqref{eq:divergence_outside_polys} holds true.
\end{enumerate}
\end{theorem}
A somewhat weaker version of Theorem~\ref{theo:zeros_imply_exp_profile} appeared in the paper of Fano, Ortolani, Van Assche~\cite{fano_ortolani_van_assche_orthogonal}. We refer to~\cite{jalowy_kabluchko_marynych_zeros_profiles_part_I} for a discussion of further related references.

\begin{remark}\label{rem:positive_zeros}
Theorem~\ref{theo:exp_profile_implies_zeros} is not directly applicable if the real-rooted polynomials $(P_n)$ are allowed to have positive roots. However, if the positive roots of the polynomials $(P_n)$ are uniformly bounded by some $A\geq 0$, Theorem~\ref{theo:exp_profile_implies_zeros} can be applied to the shifted polynomials $\hat{P}_n(x)=P_n(x+A)$. If this shifted sequence possesses an exponential profile, we can find its limiting empirical zero distribution $\hat\mu$. The limiting empirical zero distribution $\mu$ of the original sequence is then $\mu(\cdot)=\hat{\mu}(\cdot - A)$.
\end{remark}

\subsection{Logarithmic and normalized logarithmic potentials} The exponential profile $g$ is also connected to the normalized logarithmic potential of $\mu$. It is proved in \cite{jalowy_kabluchko_marynych_zeros_profiles_part_I} that under the assumptions of Theorem~\ref{theo:exp_profile_implies_zeros},
\begin{equation}\label{eq:main_polys_converse_assumption_p_n(1)}
\lim_{n\to\infty}\frac{\log P_n(1)}{n}=\sup_{\alpha\in (\underline{m},\overline{m})} g(\alpha) = :M_g
\end{equation}
and the function $-g(\cdot) + M_g$ is the Legendre transform of $u\mapsto \Psi(\eee^u)$, where
\begin{equation}\label{eq:psi_for_polys}
\Psi(t)=\int_{(-\infty,0]}\log \left(\frac{t-z}{1-z}\right) \mu({\rm d}z),\quad t>0,
\end{equation}
This means that $g(\alpha) = M_g+\inf_{u\in \R} ( \Psi(\eee^u)-\alpha u )$ for all $\alpha \in (\underline{m},\overline{m})$. The function $\Psi$ can be analytically continued to the domain $\C\bsl \mathbb(-\infty,0]$. Recall that in this domain the usual logarithmic potential $U$ of $\mu$ is defined by
$$
U(t):=\int_{(-\infty,0]}\log |t-z|\mu({\rm d}z)=\Re\left(\int_{(-\infty,0]}\log (t-z)\mu({\rm d}z)\right),\quad t\in \C\bsl \mathbb(-\infty,0],
$$
and is infinite if $\mu(\{-\infty\})>0$ (and can also be infinite even if $\mu(\{-\infty\})=0$). On the other hand, $\Psi$ is well-defined for any measure $\mu$ on $[-\infty,0]$. If $U$ is finite, then of course
$$
\Psi(t)=U(t)-U(1),\quad t\in \C\bsl \mathbb(-\infty,0].
$$
The (normalized) logarithmic potential and the Cauchy transform are related via $2\frac{\dd}{\dd t}U=2\frac{\dd}{\dd t}\Psi=G$, where $t$ is a complex variable and $\frac{\dd}{\dd t}=\frac 1 2\big(\frac{\partial}{\partial \Re(t)}-\ii\frac{\partial}{\partial \Im(t)}\big)$ is the complex (Wirtinger) derivative. Restricted to the real-valued argument $t>0$, both logarithmic potentials $\Psi$ and $U$ are antiderivatives of $G$, satisfying $\Psi^{\prime}(t)=U'(t)=G(t)$ for $t>0$.

\section{Overview of the main results}
The main contribution of this article is to provide a unified reference for the limiting empirical zero distributions of a broad class of real-rooted polynomials. While many of the individual results are known, this work brings them together under a common framework using a universal proof method based on exponential profiles.

A brief summary of our findings is presented in Table~\ref{tab:table1}. The second column lists the coefficients of $P_n$, while the third column displays the (typically simpler) function $\eee^{-g'}$, where $g$ is the exponential profile associated with the sequence $(P_n)$.\footnote{This function is closely related to the $S$-transform of $\mu$, see \cite[Lemma 6.1]{jalowy_kabluchko_marynych_zeros_profiles_part_I}, \cite{arizmendi2024s} and Equation~\eqref{eq:S_mu-via-profiles} below.} The symbol $\leftrightarrow$ indicates that the profile corresponds not directly to $(P_n)$, but to related polynomials obtained through simple transformations (like a scaling, shift or reflection), as the Hermite, Legendre and Gegenbauer polynomials do not have nonnegative coefficients. The final column presents several ways of characterizing the limiting empirical zero distribution $\mu$, such as by name, moments, Lebesgue density, or Cauchy transform. Additional details regarding the notation used, further properties and references to earlier discoveries can be found in the respective subsections.

\begin{table}\label{fig:table}\renewcommand{\arraystretch}{1.25}
\begin{tabular}{c|c|c|c}
Polynomial $P_n(z)$ & Coefficient $a_{k:n}$ & profile $\rightsquigarrow \eee^{-g'(\alpha)}$ & distribution $\mu$\\
\hline
Stirling & $\genfrac{[}{]}{0pt}{}{n}{k} n^k$ & $\frac{\alpha}{w_S(\alpha)}$
 & Unif$_{[-1,0]}$\\
\hline
Touchard $T_n(nz)$& $\genfrac{\{}{\}}{0pt}{}{n}{k} n^k$ & $\frac {1}{w_T(\alpha)\eee^{w_T(\alpha)}}$
 & $\int x ^k d\mu=\frac{(-k)^k}{(k+1)!}$\\
\hline
Fubini $F_n(z)$ & $\genfrac{\{}{\}}{0pt}{}{n}{k} k!$  & $\frac 1 {\eee^{w_T(\alpha)}-1}$ & $ \frac {dx} {|x| (x+1) \left(\pi^2 + \log^2 \left|1 + \frac 1x\right|\right)}$\\
\hline
Eulerian $E_n(z)$ & $\genfrac{\langle}{\rangle}{0pt}{}{n}{k}$ & $\eee^{w_E(\alpha)}$ & $-\log$-Cauchy\\
\hline
gen.~Narayana & $\binom{n}{k}^\gamma$ & $\left(\frac{\alpha}{1-\alpha}\right)^\gamma$ & $G_{B,\gamma}(z) = \frac 1 {z(z^{-1/\gamma}+1)}$\\
\hline
Laguerre $L_n^{(\gamma n)}(nz)$ & $(-1)^k \binom{n+\gamma}{n-k}\frac{n^k}{k!}$ & $\frac{\alpha(\alpha+\gamma)}{1-\alpha} $ &$\mathsf{MP}_{\gamma+1,\frac{1}{\gamma +1}}$\\
\hline
Hermite He$_n(\sqrt n z)$ & $\frac{(-1)^k}{k!  2^k} \cdot \frac{n^{n/2-k}}{(n-2k)!}$& $\leftrightarrow\frac{\alpha^2}{1-\alpha} $ & $\mathsf{sc}_1=\frac{dx}{2\pi}\sqrt{4-x^2}$\\
\hline
Legendre & $\binom{n}{k}\binom{\frac{n+k-1}{2}}{n}$& $\leftrightarrow\frac{\alpha^2}{1-\alpha^2}$ & arcsine$_{[-1,1]}$\\
\hline
Gegenbauer & $(-1)^{\frac{n-k}2}\frac{\Gamma(\gamma +\frac{n+k}2)}{k!(\frac{n-k}2)!}2^k$ & $\leftrightarrow\frac{\alpha(\gamma+\alpha-1/2)}{(1-\alpha)(2\gamma+\alpha)}$ & $c\frac{(b-x)(x-a)}{1-x^2}dx$\\
\hline
little $q$-Laguerre & $\prod_{j=1}^k\frac{q^{j-n}-q}{(1-aq^j)(1-q^j)}$ &	$\frac{(1-a\eee^{-\lambda \alpha})(1-\eee^{-\lambda \alpha})}{\eee^\lambda \eee^{-\lambda\alpha}-1}
$ & $p_{a|\lambda}(x){\rm d}x$, \eqref{eq:p_qa} \\
\hline
\end{tabular}
\vspace{2mm}
\caption{An overview of some explicit examples of polynomials $P_n$ in the first column and our findings on their limiting empirical zero distributions. Further details on notation used can be found in Section~\ref{sec:classical_ensembles}.}
\label{tab:table1}
\end{table}

Our general approach to filling Table~\ref{tab:table1}, which can also be applied to any sequence of real-rooted polynomials $(P_n)$ with nonpositive roots, can be briefly summarized as follows. Given such a sequence $(P_n)$, the first step is to pass to a normalized polynomial $\widetilde{P}_n(x) = b_n P_n(a_n x)$ using appropriate scaling factors $a_n$ and $b_n$. These sequences $(a_n)$ and $(b_n)$ are chosen so that $\widetilde{P}_n$ admits an exponential profile, meaning that~\eqref{eq:exp_profile_definition} holds for some function $g = g_P$. If such a normalization is possible, then Theorem~\ref{theo:exp_profile_implies_zeros} guarantees that the empirical measures of the roots of $\widetilde{P}_n$ converge to a limiting measure $\mu = \mu_P$. The second step is to characterize $\mu$ via its Cauchy transform $G$, which can be recovered from the profile $g$ by inverting the function $\alpha \mapsto \eee^{-g^{\prime}(\alpha)}$. On the third step, knowing the Cauchy transform $G$, we calculate the moments, the logarithmic potential of $\mu$ and the Lebesgue density using the Stieltjes--Perron inversion formula~\eqref{eq:Perron_inversion}. In some situations not only $t\mapsto tG(t)$ but also the Cauchy transform $G$ itself can be written as an inverse of some simple function $f$. In this case, a classic formula for the antiderivative of an inverse function
\begin{equation}\label{eq:antiderivative_inverse}
\int f^{\leftarrow}(t){\rm d}t=tf^{\leftarrow}(t)-F(f^{\leftarrow}(t))+C,
\end{equation}
can be used to calculate the logarithmic potential of $\mu$. In this formula $F$ is the antiderivative of $f$ and $f^{\leftarrow}$ denotes the inverse function to $f$.

\section{Classical polynomials}\label{sec:classical_ensembles}

\subsection{Stirling numbers of the first kind}
Consider the polynomials
\begin{equation}\label{eq:Stirling_polynomials1}
S_n(x):=x(x+1)\cdots(x+n-1)=\sum_{k=1}^{n}\stirling{n}{k}x^k,
\end{equation}
where $\stirling{n}{k}$ is the unsigned Stirling number of the first kind. Obviously, the sequence of probability measures $\lsem S_n(nx)\rsem_n$ converges weakly to the uniform distribution $\mu_S=\unif_{[-1,0]}$ on $[-1,0]$ with the Cauchy transform
$$
G_S(t)=\log \left(1+\frac{1}{t}\right),\qquad t\in  \C \bsl [-1,0].
$$
As a warm up example, let us check this trivial result by applying Theorem~\ref{theo:exp_profile_implies_zeros} to the polynomials $S_n(nz)/n^n$. By the classic logarithmic asymptotic
of the Stirling numbers of the first kind, see, for example, Eq.~(5.7) in~\cite{moser_wyman}, we have
$$
g_S(\alpha):=\lim_{n\to\infty}\frac{1}{n}\log \left(\stirling{n}{\lfloor \alpha n\rfloor} n^{\lfloor \alpha n\rfloor-n}\right)=-1+\alpha+(1-\alpha)\log \alpha + w_S+(\alpha-1)\log w_S,\quad \alpha\in (0,1),
$$
where $w_S=w_S(\alpha)$ is a unique solution in $(0,\infty)$ of the equation
\begin{equation}\label{eq:w_0_define}
\frac{w_S}{\eee^{w_S}-1}=\alpha.
\end{equation}
Taking the derivative $g^{\prime}_S$ and simplifying the expression using~\eqref{eq:w_0_define} we obtain $\eee^{-g^{\prime}_S(\alpha)}=\alpha/w_S(\alpha)=(\eee^{w_S(\alpha)}-1)^{-1}=:L_S(w_S(\alpha))$, where $L_S(t):=(\eee^t-1)^{-1}$. Recalling that $t\mapsto tG_S(t)$ is the inverse of $\alpha\mapsto \eee^{-g^{\prime}_S(\alpha)}$ we conclude that
$$
G_S(t)=t^{-1}w_S^{\leftarrow}(L_S^{\leftarrow}(t))=\log \left(1+\frac{1}{t}\right),\quad t>0.
$$

\subsection{Touchard polynomials and Stirling numbers of the second kind}
These polynomials, also called the \textit{one-variable Bell} or the \textit{exponential polynomials} are defined by
$$
T_n(x) = \sum_{k=1}^n \stirlingsec{n}{k} x^k,
$$
where $\stirlingsec{n}{k}$ is the Stirling number of the second kind. There are several equivalent definitions of Touchard polynomials including the following ones:
\begin{equation}\label{eq:touchard_def1}
T_n(x)
=
\eee^{-x} \sum_{\ell=0}^\infty \frac{\ell^n x^\ell}{\ell!}
=
n! [t^n] \eee^{x (\eee^t-1)}
=
\eee^{-x} \left(xD\right)^n \eee^x
=
B_n(x,\ldots,x),
\end{equation}
where $B_n(z_1,\ldots,z_n)$ is the $n$-th Bell polynomial and we abbreviated the derivative with $D=\frac{\dd }{\dd x}$. We refer to the books~\cite[Section~6.1]{graham_etal_book} and~\cite{mezo_book} for an introduction to Stirling numbers of both kinds.  Also, there is a Rodrigues-type formula
$$
T_n(\eee^x) = \eee^{- \eee^x} D^n \eee^{\eee^x}.
$$
Harper~\cite{harper} showed that all roots of $T_n(x)$ are real and nonpositive and used this fact to prove a CLT for the Stirling numbers of the second kind. As for the Stirling numbers of the first kind, it is natural to consider the scaled Touchard polynomials $T_n(nx)/n^n$.

The following result is originally due to Elbert~\cite{elbert1,elbert2}, see also Proposition 2.14 in~\cite{Kabluchko+Marynych+Pitters:2024}.

\begin{theorem}\label{theo:touchard_polys_zeros}
The sequence of probability measures $\lsem T_n(nx)\rsem_n$ converges weakly to a probability measure $\mu_{T}$ which is concentrated on the interval $(-\eee,0)$ and has the following
\begin{align}
&\text{Cauchy transform:}&  &G_{T}(t) = \frac {1}{t W_0(1/t)} - 1 = \eee^{W_0(1/t)}-1,  &\quad& t\in  \C \bsl [-\eee,0],
\label{eq:theo:touchard_polys_stieltjes}\\
&\text{Lebesgue density:}&     &p_{T}(x) = \frac 1 \pi \Im \eee^{ W_0\left(\frac 1x + \ii 0\right)}, &\quad& x\in (-\eee,0),\label{eq:theo:touchard_polys_density}\\
&\text{moments:}&              &\int_{-\eee}^0 x^k \mu_{T}(\dd x) = \frac{(-k)^k}{(k+1)!}, &\quad& k\in\N,\label{eq:theo:touchard_polys_moments}\\
&\text{log-potential:}&        &U_{T} (t)  = \Re\left(\frac {1}{W_0(1/t)} + W_0(1/t) - t - 1 + \log t\right), &\quad& t\in  \C \bsl [-\eee,0].
\label{eq:theo:touchard_polys_log_pot}
\end{align}
Here, $W_0$ denotes the principal branch of the Lambert $W$-function.
\end{theorem}

\begin{remark}
In the Appendix~\ref{sec:lambert_properties} we collected some properties of the Lambert function which are relevant for us. For further information on this function we refer to the excellent sources \cite{corless_etal,mezo_book_lambert_function}.
\end{remark}

\begin{proof}

By the well-known logarithmic asymptotic of the Stirling numbers of the second kind, see, for example, Eq.~(5.9) in~\cite{bender}, we have
$$
g_T(\alpha):=\lim_{n\to\infty}\frac{1}{n}\log \left(\stirlingsec{n}{\lfloor \alpha n\rfloor} n^{\lfloor \alpha n\rfloor-n}\right)=f(w_T;\alpha)+\alpha-1-\alpha\log \alpha,\quad \alpha\in (0,1),
$$
where $f(x; \alpha) = \alpha \log (\eee^x-1) - \log x$ and $w_T=w_T(\alpha)\in (0,\infty)$ is a unique solution to the equation $1 - \eee^{-w_T} = \alpha \cdot w_T$. This solution can be explicitly written via the Lambert function as follows
$$
w_T(\alpha)=\alpha^{-1}+W_0(-\alpha^{-1}\eee^{-\alpha^{-1}}).
$$
Taking the derivative, a short calculation yields
\begin{equation}\label{eq:touchard_fubini_derivation1}
g^{\prime}_T(\alpha) = \frac{\dd}{\dd \alpha} (f(w_T(\alpha); \alpha)) - \log \alpha
=
\frac{\partial f}{\partial \alpha} (w_T(\alpha); \alpha) -\log \alpha
=
\log (\eee^{w_T(\alpha)}- 1) - \log \alpha
\end{equation}
and it follows that
$$
\eee^{-g^{\prime}_T(\alpha)} = \frac {\alpha}{\eee^{w_T}-1} = \frac{\alpha}{\eee^{w_T} (1 - \eee^{-w_T})} = \frac {1}{w_T\eee^{w_T}}=\frac{1}{W_0^{\leftarrow}(w_T(\alpha))}.
$$
Inverting this function we obtain
\begin{equation}\label{eq:Toucahrd_G_T_as_inverse}
G_T(t)=t^{-1}w_T^{\leftarrow}(W_0(1/t))=\frac{1-\eee^{-W_0(1/t)}}{tW_0(1/t)}=\frac{\eee^{W_0(1/t)}-1}{tW_0(1/t)\eee^{W_0(1/t)}}=\eee^{W_0(1/t)}-1.
\end{equation}
Formula~\eqref{eq:theo:touchard_polys_density} for the density $p_T$ follows by the Stieltjes--Perron inversion~\eqref{eq:Perron_inversion}. To prove formula~\eqref{eq:theo:touchard_polys_moments} for the moments of $\mu_{T}$, we compare the formula
$$
G_{T}(t) = \int_{[-\eee,0]} \frac{\mu_T({\rm d}z)}{t-z} = \frac{1}{t} \int_{[-\eee,0]}\sum_{k=0}^{\infty}\frac{z^{k}\mu_T({\rm d}z)}{t^k} =  \sum_{k=0}^\infty \frac{1}{t^{k+1}}\int_{[-\eee,0]} z^{k}\mu_T({\rm d}z)
\qquad |t|>\eee,
$$
to the following expansion, see Eq.~\eqref{eq:taylor_exp_W_0} in the Appendix,
$$
G_{T}(t) = \eee^{W_0(1/t)}-1 =  \sum_{k=0}^\infty\frac{ (-k)^{k}}{(k+1)!} t^{-k-1},
\qquad |t|>\eee.
$$
In order to verify formula~\eqref{eq:theo:touchard_polys_log_pot}, we differentiate and compare it to $2\frac{d}{dt}U_T=G_T$. Indeed, by~\eqref{eq:W_0_der} we obtain
\begin{align*}
2\frac{\dd}{\dd t}\Re\left(\frac {1}{W_0(1/t)} + W_0(1/t) - t + \log t+C\right)=\frac {1}{t W_0(1/t)}
\end{align*}
and $C=-1$ follows from $\lim_{t\to \infty} (U_T(t)-\log t)=0$. Alternatively, the antiderivative $\int G_{T}(t){\rm d}t$ can be calculated from~\eqref{eq:antiderivative_inverse} applied with $f(t)=G_{T}^{\leftarrow}(t)=((t+1)\log(1+t))^{-1}$. Here the second equality is a consequence of~\eqref{eq:Toucahrd_G_T_as_inverse}.
\end{proof}

\begin{remark}
Edrei~\cite{edrei_zeros_succ_der} also studied zeros of Touchard polynomials (or, more precisely, of successive derivatives of $\eee^{-\eee^z}$), but in his setting there is no rescaling of the variable by $n$.
\end{remark}

\begin{figure}[t]
	\centering
	\includegraphics[width=0.31\linewidth]{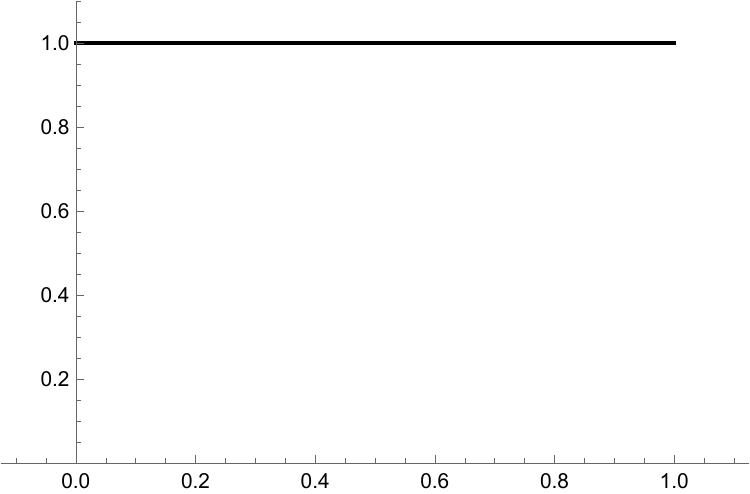}\quad \includegraphics[width=0.31\linewidth]{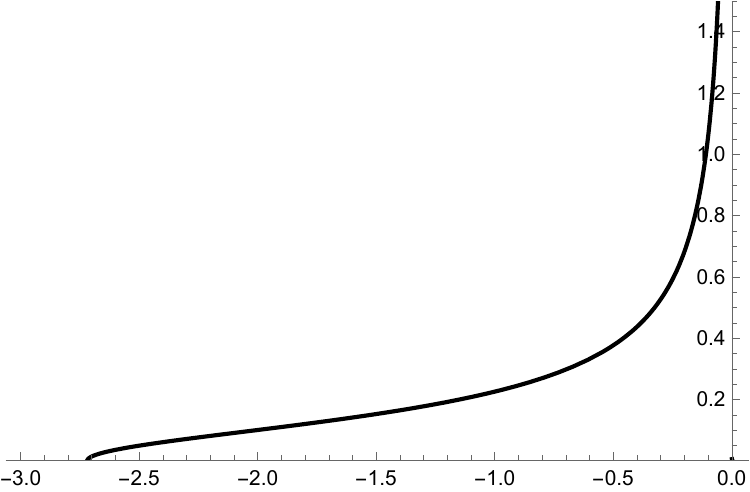}\quad \includegraphics[width=0.31\linewidth]{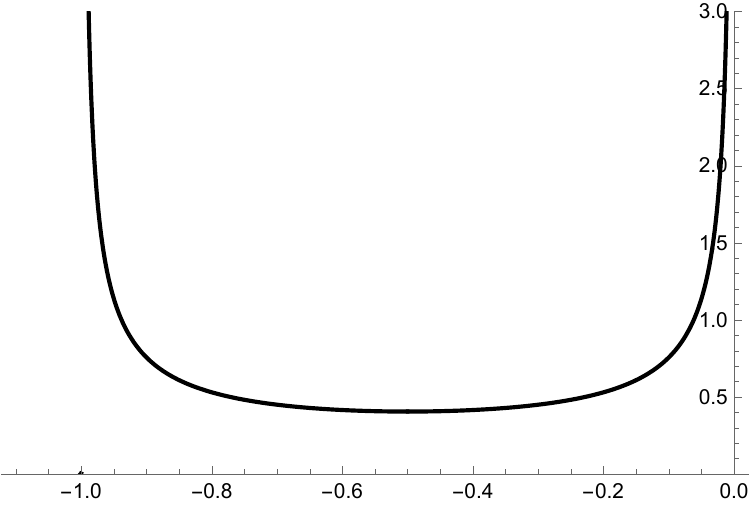}
	\caption{The limiting probability densities of zeros for generating polynomials of Stirling numbers of the first kind (left), Touchard polynomials (center) and Fubini polynomials (right).}
	\label{fig:Fubini}
\end{figure}

\subsection{Fubini polynomials}
The Fubini numbers $k!\stirlingsec{n}{k}$ count partitions of $n$ distinct elements into $k$ nonempty blocks, where the order of blocks is taken into account. Their basic properties are discussed in~\cite[Chapter~6]{mezo_book}.
Fubini polynomials are defined by
$$
F_n(x) = \sum_{k=0}^n \stirlingsec{n}{k} k! x^k.
$$
It is known~\cite[p.~147]{mezo_book} that these polynomials are real-rooted. This fact can be deduced from Rolle's theorem knowing~\cite[p. 160]{mezo_book} the formula
$$
\frac{\dd^n}{\dd x^n} \tanh x= 2^n (\tanh x - 1) F_n\left(\frac{\tanh x - 1}{2}\right).
$$

\begin{theorem}\label{theo:fubini_polys_zeros}
The sequence of probability measures $\lsem F_n(x)\rsem_n$ converges weakly to a probability measure $\mu_{F}$ on the interval $(-1,0)$ with the following
\begin{align}
&\text{Cauchy transform:}&   &G_{F}(t) = \frac 1 {t(t+1)\log (1 + \frac{1}{t})},  &\quad& t\in  \C \bsl [-1,0], \label{eq:theo:fubini_polys_stieltjes}\\
&\text{Lebesgue density:}&      &p_{F}(x) = \frac 1 {|x| (x+1) \left(\pi^2 + \log^2 \left|1 + \frac 1x\right|\right)}, &\quad& x\in (-1,0),
\label{eq:theo:fubini_polys_density}\\
&\text{moments:}&               &\int_{-1}^0 x^k \mu_{\rm{F}}(\dd x) = (-1)^k \frac{C_k}{k!}, &\quad& k\in\N_0,
\label{eq:theo:fubini_polys_moments}\\
&\text{log-potential:}&        &U_{F} (t)  = -\log\left\lvert\log\left(1+\frac{1}{t}\right)\right\rvert, &\quad& t\in  \C \bsl [-1,0].
\label{eq:theo:fubini_polys_log_pot}
\end{align}
where $C_k = \int_{0}^1 x(x+1)\ldots (x+k-1) \dd x$, $k\in \N_0$,  are the Cauchy numbers; see~\cite{merlini_sprugnoli_verri_cauchy_numbers}, \cite[pp.~293--294]{comtet_book}, \cite[Section~5.3]{mezo_book} for their properties.
\end{theorem}
\begin{proof}
We aim to apply Theorem~\ref{theo:exp_profile_implies_zeros} to the polynomials $F_n(x) /n!$ with the coefficients $\frac {k!}{n!}\stirlingsec{n}{k}$. Using the notation already introduced in the context of the Touchard polynomials, the profile is given by
$$
g_F(\alpha) = \lim_{n\to\infty}\frac{1}{n}\log \left(\stirlingsec{n}{\lfloor \alpha n\rfloor} \frac{\lfloor \alpha n\rfloor !}{n!}\right)=-f(w_T; \alpha),\quad \alpha\in (0,1).
$$
Using~\eqref{eq:touchard_fubini_derivation1} we conclude that
$$
\eee^{-g^{\prime}_F(\alpha)} = \frac{1}{\eee^{w_T}- 1}=L_S(w_T(\alpha)),
$$
where we recall the notation $L_S^{\leftarrow}(t)=\log(1+1/t)$ and $w_T^{\leftarrow}(x)=(1-\eee^{-x})/x$. Thus,
$$
G_F(t)=t^{-1}w_T^{\leftarrow}(L_S^{\leftarrow}(t))=\frac{1-\eee^{-L_S^{\leftarrow}(t)}}{tL_S^{\leftarrow}(t)}=\frac{1} {t(t+1)\log (1 + \frac{1}{t})},\quad t>0.
$$
The formula~\eqref{eq:theo:fubini_polys_density} for the density of $\mu_{F}$ follows from the Stieltjes--Perron inversion formula~\eqref{eq:Perron_inversion}. For $x\in (-1,0)$,
$$
p_{F} (x) = -\frac 1 \pi \Im G_{\rm{F}}(x + \ii 0) = -\frac 1\pi \Im \frac 1 {x(x+1)\left(\log \left|1 + \frac 1x\right| - \ii \pi\right)}
=
- \frac 1 {x (x+1) \left(\pi^2 + \log^2 |1 + \frac 1x|\right)}
$$
since
$$
\log \left(1 + \frac 1{x+\ii 0}\right) = \log \left(1 + \frac 1{x} - \ii 0\right) = \log \left|1 + \frac 1x\right| - \ii \pi.
$$
More precisely, part (vii) of~\cite[Proposition 2.1.2]{pastur_shcherbina_book}  yields the explicit formula for the density on $(-\infty,0)$ (and the existence of this density), whereas part~(v) implies the absence of the atom at $0$ taking into account that $G_{\rm{F}}(x) = o(1/x)$ as $x\uparrow 0$.

To prove formula~\eqref{eq:theo:fubini_polys_moments} for the moments of $\mu_F$, we compare the formula
$$
G_{F}(t) = \int_{[-1,0]} \frac{\mu_F({\rm d}z)}{t-z}=\sum_{k=0}^\infty \frac{1}{t^{k+1}}\int_{[-1,0]} z^{k}\mu_F({\rm d}z),
\qquad |t|>1,
$$
with the expansion
$$
G_{F}(t) = \frac 1 {t^2 \left(1+\frac{1}{t}\right)\log (1 + \frac{1}{t})} = \sum_{k=0}^\infty \frac{(-1)^k C_k}{k!} \frac 1 {t^{k+1}},
\qquad |t|>1,
$$
where we used the formula~\cite[p.~135]{mezo_book}:
$$
\frac 1 {(1+y)\log (1+y)} = \sum_{k=0}^\infty \frac {(-1)^k C_k}{k!} y^{k-1}, \qquad |y|<1.
$$
As before, the logarithmic potential in formula~\eqref{eq:theo:fubini_polys_log_pot} follows upon noticing that the Wirtinger derivative $\frac{d}{dt}$ of the right-hand side of~\eqref{eq:theo:fubini_polys_log_pot} coincides with $G_F/2$. The additive constant is recovered from the asymptotic relation $U_F(t)=\log t+o(1)$, as $t\to+\infty$.
\end{proof}

\subsection{Eulerian polynomials}
The Eulerian polynomials $(E_n)_{n\in\N}$, are defined by means of the formula
$$
\sum_{k=1}^\infty k^n x^k = \left(xD\right)^n \frac {1}{1-x} = \frac{x E_n(x)}{(1-x)^{n+1}},\quad n\in\mathbb{N}.
$$
Alternatively, if $\eulerian{n}{j}$ denotes an Eulerian number defined as the number of permutations of $1,\ldots, n$ with $n$ ascents, then
$$
E_n(x) = \sum_{j=0}^{n-1} \eulerian{n}{j} x^j,\quad n\in\mathbb{N}.
$$
For the basic properties of Eulerian numbers and polynomials we refer to~\cite{petersen_book_eulerian_numbers} and~\cite[\S~6.3--6.7]{mezo_book}. Polynomials $xE_n(x)$ are also known in the literature as Jonqui\`{e}re polynomials, see~\cite{gawronski_on_the_asymptotic} and also~\cite{gawronski_stadtmueller_jonquiere,gawronski_stadtmueller_jonquiere_real}. Theorem~\ref{theo:eulerian_polys_zeros} below is known and can be found in many papers including those of Sobolev and Sirazhdinov~\cite{sirazhdinov_euler_poly_root_distr,sobolev_eulerian_roots,sobolev_eulerian_roots_more} (see Theorem~2 in~\cite{sirazhdinov_euler_poly_root_distr}), Gawronski and Stadtm\"uller~\cite[Theorem~1]{gawronski_stadtmueller_jonquiere}, Gawronski~\cite[Theorem 5]{gawronski_on_the_asymptotic}; see also~\cite{faldey_gawronski,gawronski_stadtmueller_jonquiere_real} for further results and generalizations. It has recently been rediscovered by Melotti~\cite{Melotti}.

\begin{theorem}\label{theo:eulerian_polys_zeros}
The sequence of probability measures $\lsem E_n(x)\rsem_n$ converges weakly to a probability measure $\mu_{E}$ on the interval $(-\infty,0)$ with the following
\begin{align}
&\text{Cauchy transform:}&   &G_{E}(t) = \frac 1{t-1} - \frac{1}{t\log t},  &\quad& t\in  \C \bsl (-\infty,0],\label{eq:theo:euler_polys_stieltjes}\\
&\text{Lebesgue density:}&      &p_{E}(x) = \frac 1 {|x| (\pi^2 + \log^2 |x|)}, &\quad& x\in (-\infty, 0),\label{eq:theo:euler_polys_density}\\
&\text{Normalized log-potential:}&      &\Psi_{E}(t) = \log\left\lvert\frac{t-1}{\log t}\right\rvert, &\quad& t\in  \C \bsl (-\infty,0]\label{eq:theo:euler_polys_normal_log_potential}.
\end{align}
\end{theorem}

\begin{remark}
Formula~\eqref{eq:theo:euler_polys_density} for $p_{E}$ implies that $\mu_{E}$ is invariant under the map $x\mapsto 1/x$. If $X$ is a random variable with a scaled Cauchy distribution having density $\mathbb{R}\ni x\mapsto (\pi^2+x^2)^{-1}$, then $\mu_{E}$ is the distribution of $-\eee^{X}$, that is, up to a sign change, is a log-Cauchy distribution. All power moments of $\mu_{E}$ are infinite. In fact, even its logarithmic potential $U_{E}(t):= \int_{-\infty}^0 \log |t-x| p_{E}(x) \dd x$ equals $+\infty$. Finally, the invariance of $\mu_{E}$ under $x\mapsto 1/x$ implies the invariance of $\mu_{F}$, the weak limit of $\lsem F_n(x)\rsem_n$ in Theorem~\ref{theo:fubini_polys_zeros}, under the map $(-1,0)\ni x\mapsto -1-x\in (-1,0)$. This is also easily seen from~\eqref{eq:theo:fubini_polys_density}. Although not directly related to Theorem~\ref{theo:eulerian_polys_zeros}, the log-Cauchy distribution $\mu_E(-\cdot)$ arises as the limit of free multiplicative self-convolutions of $\boxtimes$-infinitely divisible probability measures. For further details, we refer the reader to Theorem 4.1 in~\cite{Arizmendi_Hasebe}.
\end{remark}

\begin{proof}
As in the previous examples we note that
\begin{equation}\label{eq:Euler_proof1}
g_E(\alpha):=\lim_{n\to\infty}\frac{1}{n}\log \left(\eulerian{n}{\lfloor \alpha n\rfloor} \frac{1}{n!}\right)=-\alpha w_E-\log w_E+\log (\eee^{w_E}-1),\quad \alpha\in (0,1),
\end{equation}
where $w_E=w_E(\alpha)$ is the unique solution to
\begin{equation}\label{eq:w_1_define}
\frac{\eee^{w_E}}{\eee^{w_E}-1}-\frac{1}{w_E}=\alpha,
\end{equation}
see~\cite[Eq.~(5.7)]{bender}. A straightforward differentiation of~\eqref{eq:Euler_proof1} leads to
\begin{align*}
g^{\prime}_E(\alpha)=-w_E-\alpha w'_E-w'_E\left( \frac 1 {w_E}-\frac{\eee^{w_E}}{\eee^{w_E}-1}\right)=-w_E(\alpha).
\end{align*}
Now, Theorem~\ref{theo:exp_profile_implies_zeros} implies
$$
G_E(t)=t^{-1}w_E^{\leftarrow}(\log t)=t^{-1}\left(\frac{t}{t-1}-\frac{1}{\log t}\right)=\frac 1{t-1} - \frac{1}{t\log t},\quad t>0.
$$
The normalized log-potential $\Psi_E$ is recovered by taking the antiderivative of $G_E$ and using that $\Psi_E(1)=0$.

Here is another argument which is based on the fact that the Eulerian polynomials can be expressed through the Fubini polynomials as follows:
$$
E_n(x) = (x-1)^n F_n \left(\frac 1 {x-1}\right),
$$
see~\cite[Eq.~(6.22) on p.~156]{mezo_book}. If we denote the zeros of $F_n$ by $0, x_1,\ldots, x_{n-1}$, then the zeros of $E_n$ are $y_i:= 1 + (1/x_i)$, $i=1,\ldots, n-1$. Note that $x_i \in (-1,0)$ implies that $y_i\in (-\infty, 0)$.  Therefore, if $\xi_{F}$ is a random variable whose distribution is the weak limit of $\lsem F_n(x) \rsem_n$ appearing in Theorem~\ref{theo:fubini_polys_zeros}, then $\lsem E_n(x)\rsem_n$ converges weakly to the distribution of the random variable $\xi_{E} := 1 + (1/\xi_{F})$. Its Cauchy transform is
$$
G_{E} (t) := \E \frac{1}{t - \frac 1{\xi_{F}} - 1} = \frac 1 {t-1} - \frac 1 {(t-1)^2} \E \frac 1{\frac 1 {t-1} - \xi_{F}}
=
\frac 1 {t-1} - \frac 1 {t-1}^2 G_{F}\left(\frac 1 {t-1}\right)
=
\frac 1 {t-1} - \frac 1 {t\log t},
$$
where we used the formula for $G_{F}$ from Theorem~\ref{theo:fubini_polys_zeros}. This formula is true if $1/(t-1) \notin (-1,0)$, which is equivalent to $t\notin (-\infty, 0)$.

The formula for the density  of $\xi_{\rm{E}}$ follows from the Perron inversion formula: For $x<0$,
$$
p_{E} (x) = -\frac 1 {\pi}  \Im G_{\rm{E}}(x + \ii 0) =  \frac 1 {\pi} \Im  \frac 1 {x\log (x + \ii 0)} = \frac 1 {\pi} \Im  \frac 1 {x (\log |x| + \pi \ii)} =  - \frac 1 {x (\pi^2 + \log^2 |x|)},
$$
which completes the alternative proof.
\end{proof}

\begin{figure}[t]
	\centering
	\includegraphics[width=0.31\linewidth]{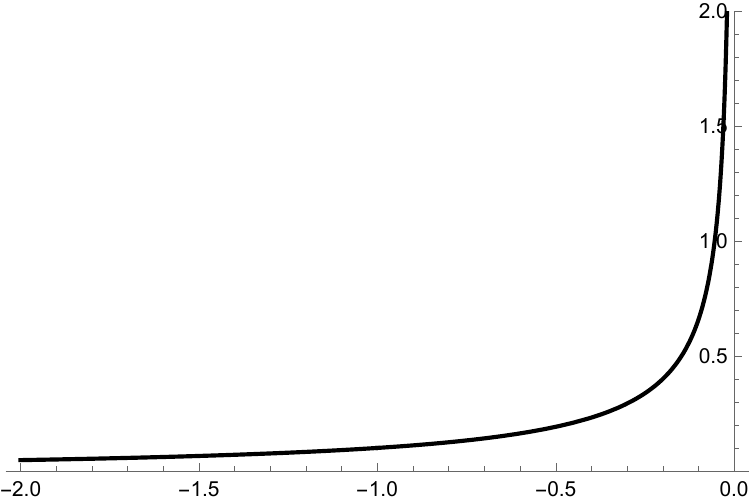}\quad \includegraphics[width=0.31\linewidth]{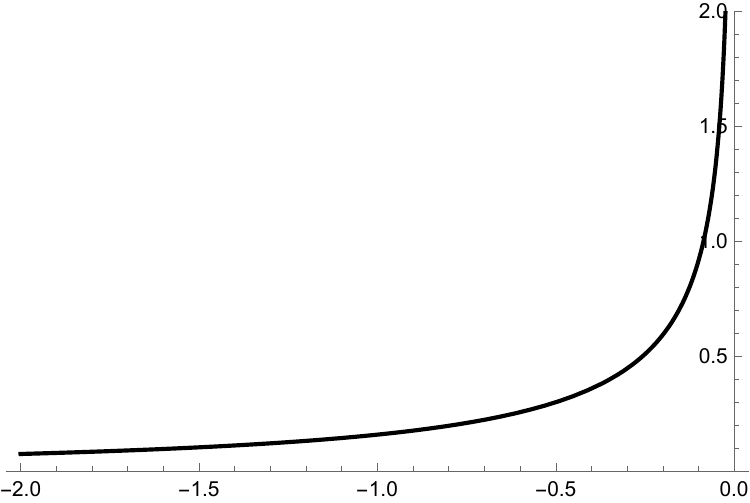}\quad \includegraphics[width=0.31\linewidth]{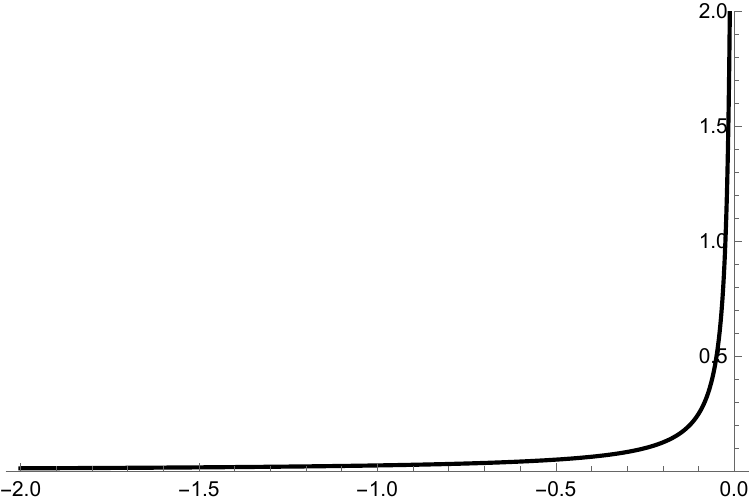}
	\caption{The limiting probability densities of zeros of Eulerian polynomials (left), Narayana polynomials for $\gamma=2$ (center) and $\gamma=10$ (right).}
	\label{fig:Narayana}
\end{figure}

\subsection{Narayana polynomials}
The Narayana numbers and the corresponding polynomials are defined by
$$
N_{n,k} = \frac{1}{k+1} \binom{n}{k} \binom{n-1}{k},
\qquad
k\in \{0, \ldots, n-1\},
\qquad
N_n(x) = \sum_{k=1}^{n} N_{n, k-1}x^k,
$$
see~\cite[Chapter~2]{petersen_book_eulerian_numbers} for their definition and properties. It is known~\cite[p. 90, Exercise 4.7]{petersen_book_eulerian_numbers} that all roots of $N_{n}$ are real (and hence nonpositive).

For any integer $\gamma\geq 2$, define also the polynomials
$$
B_{n,\gamma}(x)=\sum_{k=0}^{n} \binom{n}{k}^{\gamma}x^k,
$$
which do have only real (nonpositive) roots, see Problem 155 in Part V of Volume II~\cite{polya_book} or p.~366 in~\cite{kung_book_rota_way} (for the case $\gamma =2$, the general case follows by induction from Theorem~6.4.1 in~\cite{kung_book_rota_way}). It is easy to see that the exponential profiles of the Narayana polynomials and polynomials $B_{n,2}$ coincide.

\begin{theorem}\label{theo:narayana_polys_zeros}
For integer $\gamma\geq 2$, the sequence of probability measures $\lsem B_{n,\gamma}(x)\rsem_n$ converges weakly to a probability measure $\mu_{B,\gamma}$ on the interval $(-\infty,0)$ with the following
\begin{align}
&\text{Cauchy transform:}&   &G_{B,\gamma}(t) = \frac 1 {t(t^{-1/\gamma}+1)},  &\quad& t\in  \C \bsl (-\infty,0], \label{eq:theo:narayana_polys_stieltjes}\\
&\text{Lebesgue density:}&      &p_{B,\gamma}(x) = \frac {\sin(\pi/\gamma)} {\pi |x| (|x|^{-1/\gamma}+2\cos(\pi/\gamma)+|x|^{1/\gamma})} 
, &\quad& x\in (-\infty,0),
\label{eq:theo:narayana_polys_density}\\
&\text{log-potential:}&        &U_{B,\gamma} (t)  = \gamma\log|1+t^{1/\gamma}|, &\quad& t\in  \C \bsl (-\infty,0].
\label{eq:theo:narayana_polys_log_pot}
\end{align}
In~\eqref{eq:theo:narayana_polys_stieltjes} and~\eqref{eq:theo:narayana_polys_log_pot} we take the branch of a multi-valued function $t\mapsto t^{1/\gamma}$ which attains real values on the positive half-line. The sequence of probability measures $\lsem N_n(x)\rsem_n$ converges weakly to $\mu_{B,2}$.
\end{theorem}
\begin{proof}
The profile is given by
$$
g_{B,\gamma}(\alpha)=\lim_{n\to\infty}\frac{\gamma}{n}\log\left(\binom{n}{\lfloor n\alpha\rfloor}\right)=-\gamma(\alpha\log \alpha+(1-\alpha)\log(1-\alpha)),\quad \alpha\in (0,1).
$$
It follows that $g^{\prime}_{B,\gamma}(\alpha) =  \gamma \log (1-\alpha)-\gamma \log  \alpha$ and hence $\eee^{-g^{\prime}_{B,\gamma}(\alpha)} = (\alpha/(1-\alpha))^{\gamma}$. Inverting this function we obtain~\eqref{eq:theo:narayana_polys_stieltjes}. To prove~\eqref{eq:theo:narayana_polys_density}, we apply Perron's inversion formula. For $x<0$,
\begin{align*}
p_{B,\gamma} (x) = -\frac 1 {\pi }  \Im G_{B,\gamma}(x + \ii 0) =  -\frac 1 {\pi x} \Im  \frac{1}{1+|x|^{-1/\gamma}\eee^{-\ii\pi/\gamma}}
=\frac {|x|^{1/\gamma-1}\sin(\pi/\gamma)} {\pi (1+2\cos(\pi/\gamma)|x|^{1/\gamma}+|x|^{2/\gamma})}.
\end{align*}
Formula~\eqref{eq:theo:narayana_polys_log_pot} for the logarithmic potential follows immediately from the antiderivative of $G_{B,\gamma}(t)$.
\end{proof}

\begin{remark}
In fact, formulas~\eqref{eq:theo:narayana_polys_stieltjes}-\eqref{eq:theo:narayana_polys_log_pot} define a probability measure for all {\em real} values $\gamma>1$. This can be seen by integrating $p_{B,\gamma}$ via the change of variables $y=|x|^{1/\gamma}$, $x<0$. Also, numerical simulations suggest the following conjecture:  The polynomials $B_{n,\gamma}$ are real-rooted for every $n\in \N$ and every real $\gamma\geq 2$. On the other hand, simulations show that, for $\gamma\in (1,2)$, the polynomial $B_{n,\gamma}$ is not real-rooted if $n$ is sufficiently large.
\end{remark}

\subsection{Hypergeometric polynomials}
Let us consider a fairly general parametrized class of polynomials that covers many classical orthogonal polynomials. The results on their limiting root distribution follow those of \cite[\S 5]{martinez}, see also \cite[Theorem 3.9]{ZerosOfGenHypergeoPoly} and \cite[\S 10.3]{arizmendi2024s}.

For $i,j,n\in\N$, $\mathbf a=(a_s)_{s\le i}\in\R^i$, $\mathbf b=(b_s)_{s\le j}\in\R^j$, we define the rising and falling factorials
$$ \mathbf a ^{\overline n}:=\prod_{s=1}^i \frac{\Gamma(a_s+n)}{\Gamma(a_s)},\quad \mathbf b ^{\underline n}:=\prod_{s=1}^j \frac{\Gamma(b_s+1)}{\Gamma(b_s-n+1)}
$$
and study the (rescaled) hypergeometric polynomial
\begin{align*}
\HyperP{{\bf a}}{{\bf b}}{n}(x)&=n^{(i-j)n}\frac{(\mathbf b n)^{\underline n}}{(\mathbf a n)^{\underline n}} \ _{i+1}F_j\left(\begin{matrix}&-n,&\mathbf a n-n+1&\\
& & \mathbf b n-n+1&\end{matrix};-n^{j-i}x\right)\\
&=\sum_{k=0}^n\binom n k n^{(i-j)(n-k)} \frac{(\mathbf b n)^{\underline{n-k}}}{(\mathbf a n)^{\underline{n-k}} }	x^k.
\end{align*}
Here we used the notation
\begin{equation}\label{eq:hypergoemetric_function_def}
_{i+1}F_j\left(\begin{matrix}&-n,&\mathbf a &\\
& & \mathbf b&\end{matrix};x\right):=\sum_{k=0}^{n}\frac{(-n)^{\overline{k}}({\bf a})^{\overline{k}}}{({\bf b})^{\overline{k}}}\frac{x^k}{k!}=\sum_{k=0}^{n}\binom{n}{k}\frac{(a_1)^{\overline{k}}\cdots(a_i)^{\overline{k}}}{(b_1)^{\overline{k}}\cdots(b_j)^{\overline{k}}}(-x)^k,
\end{equation}
for a (generalized) hypergeometric polynomial. If for some choice of parameter ${\bf a}$ and ${\bf b}$, polynomials $\HyperP{{\bf a}}{{\bf b}}{n}$ have positive coefficients and nonpositive roots, at least for all large enough $n$, then we describe their limiting root distribution as follows.

\begin{theorem}\label{theo:hypergeo_polys_zeros}
Let $\mathbf a \in\R^i,\mathbf b\in\R^j$ be such that:
\begin{itemize}
\item[(i)] $a_s,b_s\notin [0,1)$;
\item[(ii)] the number of negative reals among $\{a_1,\ldots,a_i,b_1,\ldots,b_j\}$ is even;
\item[(iii)] $\HyperP{{\bf a}}{{\bf b}}{n}$ has only real nonpositive roots for all $n$ large enough.
\end{itemize}
Then, $\alpha\mapsto \frac \alpha {1-\alpha} \frac{\prod_{s=1}^j (b_s -1+\alpha)}{\prod_{s=1}^i (a_s -1+\alpha)}$ is positive and invertible on $(0,1)$ with inverse $\Phi^{{\bf a},{\bf b}}$ and $\lsem \HyperP{{\bf a}}{{\bf b}}{n}\rsem_n$ converges weakly to a distribution with the Cauchy transform $G^{{\bf a},{\bf b}}(t):=\Phi^{{\bf a},{\bf b}}(t)/t$.
\end{theorem}

For an overview on choices of the parameters $\mathbf a ,\mathbf b$ which lead to real, nonpositive roots of hypergeometric polynomials, we refer to \cite[\S 4]{martinez}.

\begin{proof}
Another application of Stirling's formula shows that the profile is given by
\begin{align*}
g_H(\alpha)&=\lim_{n\to\infty}\frac 1 n \log\left( \binom n {\lfloor\alpha n\rfloor}  \frac{n^{-j(n-\lfloor\alpha n\rfloor)}(\mathbf b n) ^{\underline{n-\lfloor\alpha n\rfloor}}}{n^{-i(n-\lfloor\alpha n\rfloor)}(\mathbf a n)^{\underline{n-\lfloor\alpha n\rfloor}} }\right)\\
&=-\alpha \log\alpha -(1-\alpha)\log(1-\alpha)- \lim_{n\to\infty} \frac 1 n \sum_{m=0}^{n-\lfloor\alpha n\rfloor-1}\log\left( \frac{\prod_{s=1}^i (a_s -\frac m n)}{\prod_{s=1}^j (b_s -\frac mn)}\right)\\
&=-\alpha \log\alpha -(1-\alpha)\log(1-\alpha) -\int_0^{1-\alpha}\log\left( \frac{\prod_{s=1}^i (a_s -x)}{\prod_{s=1}^j (b_s -x)}\right){\rm d}x,\quad \alpha\in (0,1),
\end{align*}
by a Riemann summation in the last step. Note that the ratio under the last logarithm ought to be positive and bounded away from zero for all $x\in [0,1-\alpha]$ by the assumptions (i) and (ii). The derivative of $g_H$ is given by
\begin{align*}
g_H'(\alpha)&=\log\left(\frac {1-\alpha}{\alpha } \frac{\prod_{s=1}^i (a_s -1 +\alpha)}{\prod_{s=1}^j (b_s -1+\alpha)}\right),
\end{align*}
hence the inverse of $\eee^{-g_H^{\prime}(\alpha)}=\frac \alpha {1-\alpha} \frac{\prod_{s=1}^j (b_s -1+\alpha)}{\prod_{s=1}^i ( a_s -1+\alpha)}$ is equal to $tG^{{\bf a},{\bf b}}(t)$ by Theorem~\ref{theo:exp_profile_implies_zeros}.
\end{proof}
\begin{remark}
The positivity and invertibility of the rational function $\alpha\mapsto \frac \alpha {1-\alpha} \frac{\prod_{s=1}^j (b_s -1+\alpha)}{\prod_{s=1}^i ( a_s -1+\alpha)}$ on $(0,1)$ is a necessary condition for $\HyperP{{\bf a}}{{\bf b}}{n}$ to have nonpositive roots for all large enough $n\in\mathbb{N}$.
\end{remark}

Using Theorems 4.6 and 4.7 in~\cite{martinez} we obtain the following corollary.

\begin{corollary}\label{cor:real_zeros_hypergeo_polys}
The polynomials $\HyperP{{\bf a}}{{\bf b}}{n}$ have only nonpositive roots for all $n\geq 1$ and the sequence of probability measures $\lsem \HyperP{{\bf a}}{{\bf b}}{n}\rsem_n$ converges weakly to a probability measure whose Cauchy transform is as described in Theorem~\ref{theo:hypergeo_polys_zeros} in the following scenarios:
\begin{itemize}
\item[(a)] $i,j\geq 0$, $a_1,\ldots,a_i<0$, $b_1,\ldots,b_j\geq 1$ and $i$ is even;
\item[(b)] $j\geq i$, $b_1,\ldots,b_j\geq 1$ and $a_s\geq b_s+1$, $s\in\{1,2,\ldots,i\}$;
\end{itemize}
\end{corollary}

\begin{remark}\label{rem:sequences_ab}
Observe that, if $a_s,b_s\notin [0,1)$ are replaced by sequences $a_{s}^{(n)},b_{s}^{(n)}\notin [0,1)$ converging to $a_s,b_s$, respectively, then the proof of Theorem~\ref{theo:hypergeo_polys_zeros} remains the same and the resulting weak convergence continues to hold for $\lsem \HyperP{{\bf a^{(n)}}}{{\bf b^{(n)}}}{n}\rsem_n$. Indeed, the Riemann approximation in the proof of Theorem~\ref{theo:hypergeo_polys_zeros} persists, since assumption (i) implies uniform (Lipschitz) equicontinuity of the integrand. In this case, also the statement of Corollary~\ref{cor:real_zeros_hypergeo_polys} continues to hold if $b_s \ge 1$ is replaced by $b_s^{(n)}>1-1/n$.
Actually, we shall see below that this condition can further be weakened to $b_s^{(n)}\ge 1-1/n$, by setting $n^{\underline n}=0$ or formally allowing for $\Gamma(0)=\infty$.
We are going to use these facts in the special cases to follow.
\end{remark}

\begin{corollary}\label{cor:real_zeros_hypergeo_polys_2}
In the setting of Corollary~\ref{cor:real_zeros_hypergeo_polys}(a) with $i=0$, the following holds true. Assume that $b_s^{(n)}\ge 1-1/n$, $s=1,\ldots,j$ and $\lim_{n\to\infty}b_s^{(n)}=b_s\geq 1$, $s=1,\ldots,j$. Then the polynomials $\HyperP{\emptyset}{\mathbf b^{(n)}}{n}$ still have nonpositive roots and $\lsem \HyperP{\emptyset}{\mathbf b^{(n)}}{n}\rsem _n$ converges weakly to a probability measure with the Cauchy transform $G^{\varnothing,{\bf b}}$.
\end{corollary}

\begin{proof}
Let $\mathbf b^{(n)}\in[1-1/n,\infty)^j$ and set $b^{(n)}_0=1$. Rewrite the hypergeometric polynomial
\begin{align*}
\HyperP{\emptyset}{\mathbf b^{(n)}}{n}(x)= \sum_{k=0}^n\binom n k n^{-j(n-k)} {(\mathbf b^{(n)} n)^{\underline{n-k}}} 	x^{k}=\sum_{k=0}^n \frac {n^k} {k!} \prod_{s=0}^j \frac{\Gamma(b_s^{(n)}n+1)}{n^k\Gamma(b_s^{(n)}n-k+1)}	x^{n-k}.
\end{align*}
Assume now that $b_{s_0}^{(n)}=1-1/n$ for some $s_0$, then
\begin{align*}
\HyperP{\emptyset}{\mathbf b^{(n)}}{n}(x)&=x\left(\frac{n-1}{n}\right)^{(n-1)j}\sum_{k=0}^{n-1} \frac {(n-1)^k} {k!} \prod_{s=0}^j \frac{\Gamma(b_s^{(n)}n+1)}{(n-1)^k\Gamma(b_s^{(n)}n-k+1)} \left(\Big(\frac{n}{n-1}\Big)^{j}x\right)^{n-1-k}\\
&=x\left(\frac{n-1}{n}\right)^{(n-1)j} \HyperP{\emptyset}{\tilde{\mathbf{b}}^{(n)}}{n-1} \left(\Big(\frac{n}{n-1}\Big)^{j}x\right)
\end{align*}
where $\tilde b_s^{(n)}=b_s^{(n)}(1+\frac 1 {n-1})$ and the new vector $\tilde{\mathbf{b}}^{(n)}$ consists of $j-1$ components $\tilde b_s^{(n)}$ for $s=1,\dots,j$, $s\neq s_0$. One may easily check that $\tilde b_s^{(n)}>1-\frac 1 {n-1}$ and $\tilde b_{s_0}^{(n)}=1$, hence $\HyperP{\emptyset}{\tilde{\mathbf{b}}^{(n)}}{n-1} $ has nonpositive roots with limiting distribution as given in Corollary \ref{cor:real_zeros_hypergeo_polys}. Thus, the roots of $\HyperP{\emptyset}{\mathbf b^{(n)}}{n}$ are nonpositive as well and have the same limit distribution as that of $\HyperP{\emptyset}{\tilde{\mathbf{b}}^{(n)}}{n-1} $.
\end{proof}

\subsection{Laguerre polynomials}
Consider the associated Laguerre polynomials with parameter $\gamma\in\R$:
$$
L_n^{(\gamma)} (x) = \sum_{k=0}^n (-1)^k \binom{n+\gamma}{n-k} \frac{x^k}{k!}.
$$
The Laguerre polynomial $L_n^{(n\gamma_*)}(-nx)=\frac{n^n}{n!}\HyperP{\varnothing}{1+\gamma_*}{n}(x)$ is a rescaled hypergeometric function with $i=0$, $j=1$ and ${\bf b}=1+\gamma_*$. Recall that the Marchenko--Pastur law with parameters $\lambda>0$ and $\sigma^2>0$ is given by
\begin{equation}\label{eq:MP_def}
\mathsf{MP}_{\sigma^2,\lambda}({\rm d}x)=\lambda^*\delta_0({\rm d} x)+\frac 1 {2\pi\sigma^2}\frac{\sqrt{(\lambda_+-x)(x-\lambda_-)}}{\lambda x}\ind_{(\lambda_{-},\lambda_{+})}(x)\dd x,
\end{equation}
where $\lambda_{\pm}:=\sigma^2(1\pm\sqrt \lambda)^2$ and the mass at $0$ is $\lambda^{*}:=\max(1-\lambda^{-1},0)$. It is known that the empirical root distribution of rescaled Laguerre polynomials converge to a Marchenko--Pastur law $\mathsf{MP}_{\sigma^2,\lambda}$ with appropriately chosen parameters $\sigma^2$ and $\lambda$, see~\cite[Theorem 3.1]{dette_studden} or~\cite[Theorem 5]{gawronski_strong_laguerre_hermite}. According to~\cite[Lemma 3.11]{BaiSilverstein} the Cauchy transform is equal to
\begin{equation}\label{eq:marchenko_pastur_Stieltjes}
G_{\mathsf{MP}_{\sigma^2,\lambda}}(t)=\frac {t-\sigma^2(1-\lambda)-\sqrt{(\lambda_{+}-t)(\lambda_{-}-t)}}{2\sigma^2\lambda t},\quad t\in  \C \bsl [\lambda_-,\lambda_+],
\end{equation}
where the branch is chosen such that $t\mapsto \sqrt{(\lambda_{+}-t)(\lambda_{-}-t)}\sim t$, as $|t|\to+\infty$.

\begin{corollary}\label{cor:laguerre_polys_zeros}
Let $(\gamma_n)_{n\in \N}$ with $\gamma_n\in [0,\infty)$ or $\gamma_n\in\{-1,\dots,-n+1\}$ be a sequence such that $ \gamma_*:=\lim_{n\to\infty} \frac 1n \gamma_n  >-1$. Then $\lsem L_n^{(\gamma_n)}(nx)\rsem_n$ converges weakly to a Marchenko--Pastur law with $\sigma^2=\gamma_*+1$, $\lambda=\frac{1}{\gamma_*+1}$ and $\lambda_{\pm}=\gamma_*+2\pm 2\sqrt{\gamma_*+1}$.
\end{corollary}

\begin{proof}
Let us first consider the case $\gamma_n\ge 0$.
By Theorem~\ref{theo:hypergeo_polys_zeros}, the exponential profile $g$ satisfies $\eee^{-g'(\alpha)}=\frac{\alpha(\alpha+\gamma_*)}{1-\alpha}$  and $\lsem L_n^{(\gamma_n)}(-nx)\rsem_n$ converges weakly to a distribution with the Cauchy transform
\begin{align*}
G(t):=G^{\varnothing,1+\gamma^{\ast}}(t)=\frac {-\gamma_*-t+\sqrt{\gamma_*^2+2(\gamma_*+2)t+t^2}}{2t},
\end{align*}
where the branch of $t\mapsto \sqrt{\gamma_*^2+2(\gamma_*+2)t+t^2}$ is chosen such that it is asymptotically equivalent to $t$, as $|t|\to+\infty$. This choice ensures that $G(t)\sim 1/t$, as $|t|\to+\infty$. Comparing the above formula to~\eqref{eq:marchenko_pastur_Stieltjes} reveals that $G$ is the Stieltjes transform of $\mathsf{MP}_{1+\gamma_{\ast},(1+\gamma_{\ast})^{-1}}(-\cdot)$.

Assume now that $-\gamma_{n}\in\{1,2,\ldots,n-1\}$ and $\gamma_n/n\to \gamma^{\ast}\in (-1,0]$. Using the relation
	$$
	L_n^{(\gamma_n)}(x)=(-x)^{-\gamma_n}\frac{(n+\gamma)!}{n!}L_{n+\gamma_n}^{(-\gamma_n)}(x)
	$$
	we conclude that the roots of $L_n^{(\gamma_n)}(x)$ are nonnegative and
	$$
	\lsem L_n^{(\gamma_n)}(nx)\rsem_n=\lsem x^{-\gamma_n}L_{n+\gamma_n}^{(-\gamma_n)}(nx)\rsem_n=\frac{-\gamma_n}{n}\delta_0+\frac{n+\gamma_n}{n}\left\lsem L_{n+\gamma_n}^{(-\gamma_n)}\left(\frac{n}{n+\gamma_n}(n+\gamma_n)x\right)\right\rsem_{n+\gamma_n}.
	$$
	According to the case $\gamma_n\ge 0$, we have
	$$
	\lsem L_n^{(\gamma_n)}(nx)\rsem_n\toweak -\gamma^{\ast}\delta_0+(1+\gamma^{\ast})\mathsf{MP}_{(1+\gamma^{\ast})^{-1},1+\gamma^{\ast}}\left((1+\gamma^{\ast})^{-1}(\cdot)\right)=\mathsf{MP}_{1+\gamma_{\ast},(1+\gamma_{\ast})^{-1}}(\cdot),
	$$
	where the last step follows from a straightforward calculation of the rescaled density \eqref{eq:MP_def}.
\end{proof}

For completeness, let us derive the logarithmic potential of ${\mathsf{MP}_{\sigma^2,\lambda}}$. It is easy to see, that $G_{\mathsf{MP}_{\sigma^2,\lambda}}$ is a solution to the quadratic equation
$$
(\sigma^2\lambda t)x^2-(t-\sigma^{2}(1-\lambda))x+1=0
$$
and thereupon is an inverse of the function
\begin{equation}\label{eq:inverse_G_MP}
x\mapsto -(1+\sigma^2(1-\lambda)x)/(x(\sigma^2\lambda x-1))=\frac{1}{x}-\frac{\sigma^2}{\sigma^2\lambda x-1}.
\end{equation}
on an appropriate interval of monotonicity (whose exact form is of no importance for our purposes). Applying~\eqref{eq:antiderivative_inverse} with $f$ given by~\eqref{eq:inverse_G_MP} we conclude that
$$
\int G_{\mathsf{MP}_{\sigma^2,\lambda}}(t){\rm d}t=tG_{\mathsf{MP}_{\sigma^2,\lambda}}(t)-\log |G_{\mathsf{MP}_{\sigma^2,\lambda}}(t)|+\lambda^{-1}\log |1-\sigma^2\lambda G_{\mathsf{MP}_{\sigma^2,\lambda}}(t)|+C_{\lambda,\sigma^2}.
$$
Thus, for large enough positive real $t$,
\begin{equation}\label{eq:log_potential-MP}
\int_{\R}\log (t-x)\mathsf{MP}_{\sigma^2,\lambda}({\rm d}x)=tG_{\mathsf{MP}_{\sigma^2,\lambda}}(t)-\log |G_{\mathsf{MP}_{\sigma^2,\lambda}}(t)|+\lambda^{-1}\log |1-\sigma^2\lambda G_{\mathsf{MP}_{\sigma^2,\lambda}}(t)|-1,
\end{equation}
where the constant $C_{\lambda,\sigma^2}=-1$ was determined using $tG_{\mathsf{MP}_{\sigma^2,\lambda}}(t)\to 1$ and the fact that the right-hand side of the last centered formula must satisfy $\log t+o(1)$, as $t\to+\infty$.

\begin{figure}[t]
	\centering
	\includegraphics[width=0.23\linewidth]{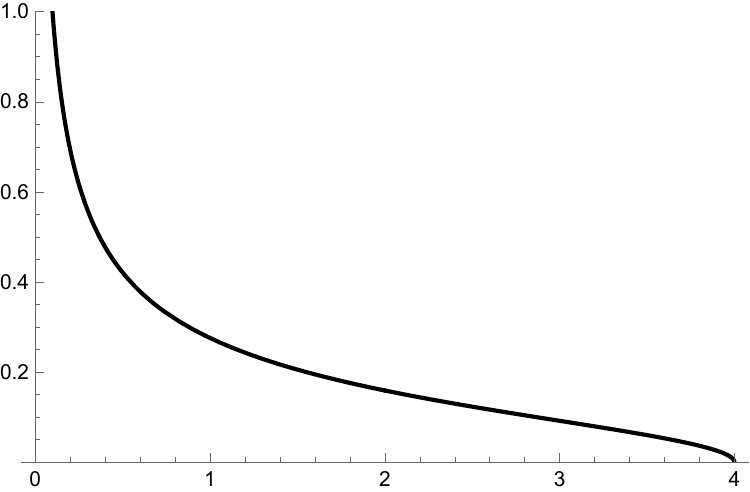}\quad \includegraphics[width=0.23\linewidth]{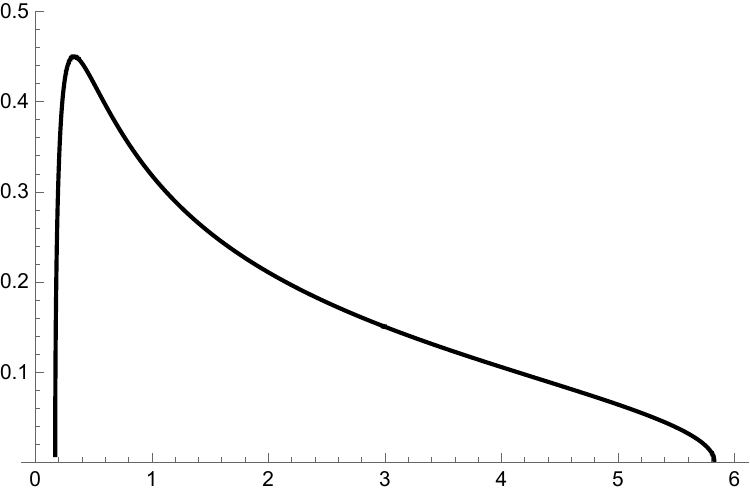}\quad \includegraphics[width=0.23\linewidth]{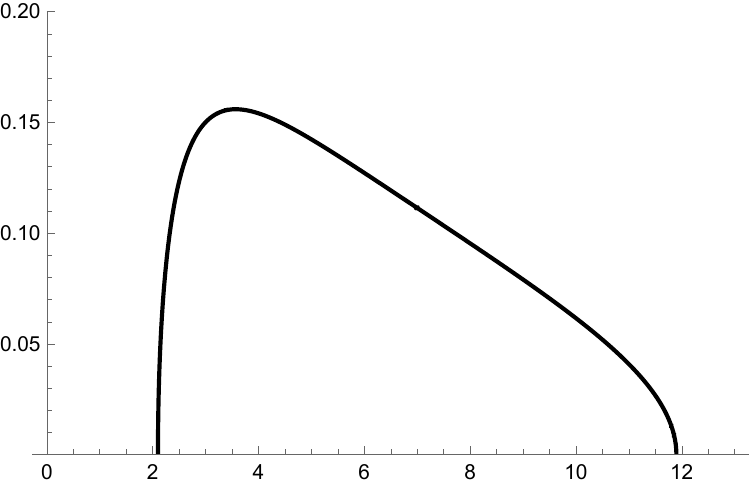}\quad \includegraphics[width=0.23\linewidth]{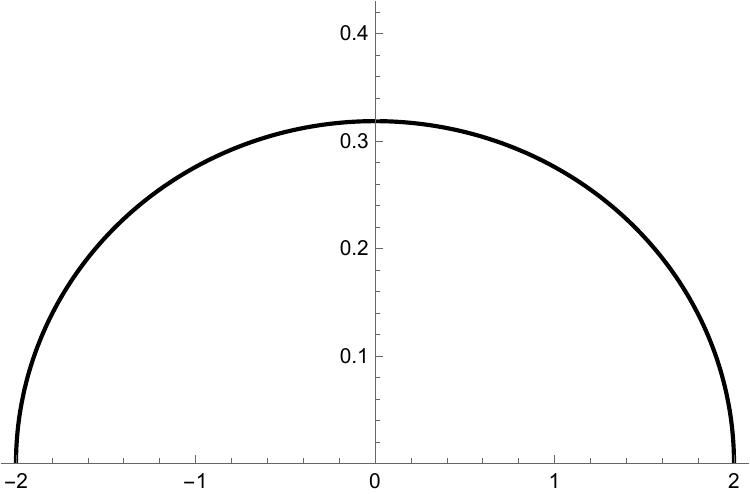}
	\caption{The limiting probability densities of zeros of Laguerre polynomials for $\gamma^{\ast}\in\{0,1,5\}$ (from left to right) and Hermite polynomials (right).}
	\label{fig:Laguerre}
\end{figure}

\subsection{Hermite polynomials}

The classical (probabilist) Hermite polynomials are given by
\begin{align}\label{eq:Hermite}
\mathrm{He}_n(x)
=
\eee^{-\frac 1 2  D^2} x^n
=
n! \sum_{m=0}^{\lfloor n/2\rfloor}  \frac{(-1)^m }{m!  2^m} \cdot \frac{x^{n-2m}}{(n-2m)!},
\end{align}
where we used the notation $D=\frac{\dd}{\dd x}$. It is known that Hermite polynomials have only real roots and these roots are symmetric around the origin. In particular, they are not all nonpositive and Theorem~\ref{theo:exp_profile_implies_zeros} is not directly applicable. Nevertheless, using the following relation between Hermite and Laguerre polynomials
\begin{align}\label{eq:Hermite-Laguerre}
\mathrm{He}_n(x)=(-2)^{\lfloor \tfrac n 2\rfloor}{\lfloor \tfrac n 2\rfloor}!x^{\epsilon} L_{\lfloor \tfrac n 2\rfloor}^{(\epsilon-1/2)}(x^2/2),
\end{align}
where $\epsilon=1$ if $n$ is odd and $\epsilon=0$ if $n$ is even, we can recover the well known semicircle distribution as the limiting root distribution of Hermite polynomials. Indeed, formula~\eqref{eq:Hermite-Laguerre} and Corollary~\ref{cor:laguerre_polys_zeros} applied with $\gamma_n=\epsilon-1/2$ and $\gamma^{\ast}=0$, imply
\begin{align*}
\lsem \mathrm{He}_n( \sqrt n x)\rsem_n
&=\lsem L_{\lfloor \tfrac n 2\rfloor}^{(\epsilon-1/2)}(n x^2/2)\rsem_n+\epsilon\tfrac{1}{n}\delta_0\\
&=n^{-1} \lfloor \tfrac n 2\rfloor\lsem L_{\lfloor \tfrac n 2\rfloor}^{(\epsilon-1/2)}(n x^2/2)\rsem_{\lfloor \tfrac n 2\rfloor}+\epsilon\tfrac{1}{n}\delta_0
\toweak \frac{1}{2}\big(\mathsf{MP}_{1,1}(-(\cdot)^2)+\mathsf{MP}_{1,1}((\cdot)^2)\big)=\mathsf{sc}_1.
\end{align*}
Here, $\mathsf{sc}_1$ is the standard semi-circle distribution with the density~\eqref{eq:theo:hermite_polys_density} and the last line follows from a direct calculation using that $\mathsf{MP}_{1,1}((\cdot)^2)$ is the quarter-circular law.
Thus, we obtain the following
\begin{corollary}\label{cor:hermite_polys_zeros}
The weak limit of $\lsem \mathrm{He}_n( \sqrt n x)\rsem_n$ is the standard semicircle distribution $\mathsf{sc}_1$ with
\begin{align}
&\text{Cauchy transform:}&   &G_{\mathsf{sc}_1}(t)=\frac 1 2 \big( t-\sqrt{t^2-4}),  &\quad& t\in  \C \bsl [-2,2], \label{eq:theo:hermite_polys_stieltjes}\\
&\text{Lebesgue density:}&      &p_{\mathsf{sc}_1}(x)=\frac 1 {2\pi}\sqrt{4-x^2}
, &\quad& x\in [-2,2],
\label{eq:theo:hermite_polys_density}\\
&\text{log-potential:}&        &U_{\mathsf{sc}_1}(t)=\frac 14 \Re\left(t^2 - t \sqrt{t^2 - 4}\right)  + \log\left|\frac{t + \sqrt{t^2 - 4}}{2}\right| -\frac 12,
 &\quad& t\in  \C .
\label{eq:theo:hermite_polys_log_pot}
\end{align}
\end{corollary}
\begin{proof}
The convergence to $\mathsf{sc}_1$ has already been proved. Formula~\eqref{eq:theo:hermite_polys_log_pot} follows  from~\eqref{eq:antiderivative_inverse} by noting that $G_{\mathsf{sc}_1}$ is a solution to the equation $x^2-tx+1=0$, hence is an inverse of $(-\infty,-1)\ni x\mapsto x+x^{-1}\in (-\infty,-2)$. Thus,
\begin{multline*}
\int \log(t-x)\mathsf{sc}_1({\rm d}x)=\int G_{\mathsf{sc}_1}(t){\rm d}t\overset{\eqref{eq:antiderivative_inverse}}{=}tG_{\mathsf{sc}_1}(t)-\frac{1}{2}G_{\mathsf{sc}_1}^2(t)-\log |G_{\mathsf{sc}_1}(t)|+C\\
=\frac{1}{2}tG_{\mathsf{sc}_1}(t)+\frac{1}{2}-\log |G_{\mathsf{sc}_1}(t)|+C=\frac{1}{4}(t^2-t\sqrt{t^2-4})+\frac{1}{2}+\log \left|\frac{t+\sqrt{t^2-4}}{2}\right|+C.
\end{multline*}
The constant $C=-1$ is determined from $U_{\mathsf{sc}_1}(t)=\log t + o(1)$, as $t\to+\infty$.
\end{proof}

\begin{remark}
As we have already mentioned, Theorem~\ref{theo:exp_profile_implies_zeros} is not directly applicable because
$\mathrm{He}_n$ have also positive roots. However, the semicircle limiting distribution for $\mathrm{He}_n$ could have been derived directly (without a resort to Laguerre polynomials) by using a trick discussed in Remark~\ref{rem:positive_zeros} in conjunction with the classic upper bound for zeros of $\mathrm{He}_n$, see Theorem 6.32 in~\cite{Szego}.
\end{remark}

\subsection{Jacobi polynomials}
For $u,v>-1$ the Jacobi polynomials $J_n^{(u,v)}$ are defined by
$$
J_n^{(u,v)}(x)=\frac{(u+1)^{\overline{n}}}{n!} \ _{2}F_{1} \left(\begin{matrix}&-n,&1+u+v+n&\\
& & u+1 &\end{matrix};\frac{1-x}{2}\right),\quad n\in\mathbb{N}.
$$
In terms of the hypergeometric polynomials this implies (with $i=j=1$)
$$
n! \frac{\Gamma((u+v+1)n+1)}{\Gamma((u+v+2)n+1)} J_n^{(u n,v n)}(2x+1)=\HyperP{2+u+v}{1+u}{n}(x).
$$
Thus, part (b) of Corollary~\ref{cor:real_zeros_hypergeo_polys} and the remark after it yield that, for any sequences $u_n,v_n\geq 0$ such that $u_n/n\to u\geq 0, v_n/n\to v\geq 0$,
$$
\lsem J_n^{(u_n,v_n)}\rsem_n \toweak \mu_{J^{(u,v)}},
$$
where $\mu_{J^{(u,v)}}$ is the pushforward under $x\mapsto 2x+1$ of the probability measure $\hat{\mu}_{J^{(u,v)}}$ with the Cauchy transform
\begin{equation*}
G_{\hat{\mu}_{J^{(u,v)}}}(t)=\frac{-t(u+v)-u+\sqrt{(t(u+v)+u)^2+4(1+u+v)t(1+t)}}{2t(1+t)}.
\end{equation*}
The above formula follows by solving the quadratic equation
$$
\eee^{-g'(\alpha)}=t\quad \Longleftrightarrow\quad \frac{\alpha(u+\alpha)}{(1-\alpha)(1+u+v+\alpha)}=t\quad \Longleftrightarrow\quad (1+t)\alpha^2+(t(u+v)+u)\alpha-t(1+u+v)=0$$
and picking out of two solutions the one satisfying $G_{\hat{\mu}_{J^{(u,v)}}}(t)\sim t^{-1}$ as $t\to+\infty$. Therefore,
\begin{equation}\label{eq:Stieltjes_Jacobi}
G_{\mu_{J^{(u,v)}}}(t)=\frac{1}{2}G_{\hat{\mu}_{J^{(u,v)}}}\left(\frac{t-1}{2}\right)=\frac{-t(u+v)-(u-v)+\sqrt{(t(u+v)+(u-v))^2+4(1+u+v)(t^2-1)}}{t^2-1}.
\end{equation}
From this equation one case deduce the density of $\mu_{J^{(u,v)}}$ by an appeal to Stieltjes-Perron formula~\eqref{eq:Perron_inversion}. We omit these standard calculations and refer to~\cite{dette_studden} or Section 5.3 in~\cite{martinez}.

The special case $u=v=\gamma-1/2$ corresponds to the Gegenbauer polynomials with parameter $\gamma\geq 1/2$, then
\begin{equation}\label{eq:Stieltjes_Gegenbauer}
	G_{\mu_{J^{(\gamma-1/2,\gamma-1/2)}}}(t)=\frac{-t(2\gamma-1)+\sqrt{t^2(2\gamma-1)^2+8\gamma(t^2-1)}}{t^2-1}.
\end{equation}
The case $u=v=0$ corresponds to the Legendre polynomials with
\begin{equation}\label{eq:Stieltjes_Legendre}
G_{\mu_{J^{(0,0)}}}(t)=\frac{2}{\sqrt{t^2-1}},\quad t\notin [-1,1],
\end{equation}
which means that $\mu_{J^{(0,0)}}$ is the arcsine distribution with the density $x\mapsto 2\pi^{-1}(1-x^2)^{-1/2}$, $x\in [-1,1]$.

\begin{figure}[t]
	\centering
	\includegraphics[width=0.33\linewidth]{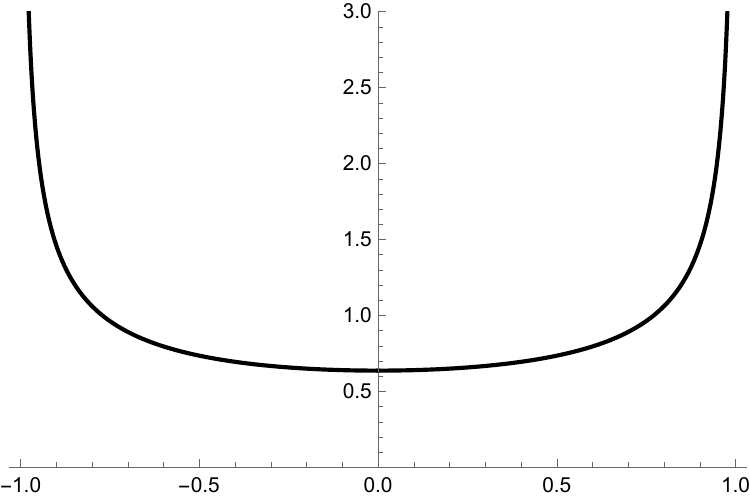}\includegraphics[width=0.33\linewidth]{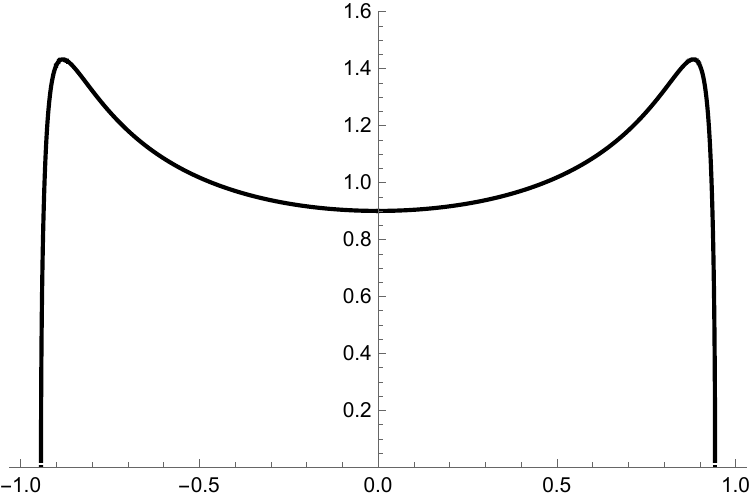}\includegraphics[width=0.33\linewidth]{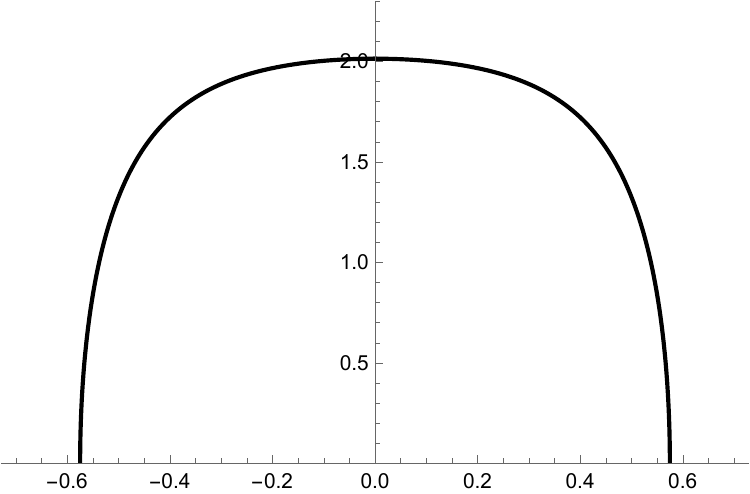}
	\caption{The limiting distribution of zeros of Legendre polynomials (left, $\gamma=1/2$), Gegenbauer polynomials for $\gamma=1$ (center) and for $\gamma=5$ (right).}
	\label{fig:Gegenbauer}
\end{figure}

\subsection{Little \texorpdfstring{$q$}{q}-Laguerre polynomials}
The $q$-analogues of classical orthogonal polynomials can be interpreted as quantized versions of these polynomials, where continuous variables are replaced by discrete ones taking values in a $q$-geometric lattice $\{q^k\}_{k\geq 0}$, and derivatives become by $q$-difference operators. We refer to the monograph~\cite{KLS} for an exhaustive overview of all such $q$-orthogonal polynomials, which are real rooted with roots in the support of the weight function. Their empirical limiting zero distribution can be obtained using exponential profiles, but we restrict ourselves to one particularly illustrative family. One possible $q$-deformed model of Laguerre polynomials leads to the \emph{little $q$-Laguerre polynomials} $P_n(x;a|q)$, whose definition involves `$q$-factorials' and `$q$-hypergeometric series' ${} _2\phi_1$, and is explicitly given by
\begin{align*}
	P_n(-x;a|q)&=_{2}\phi_1\left(\begin{matrix}q^{-n} & 0 \\ \multicolumn{2}{c}{aq} & \end{matrix};q,-qx\right)=\sum_{k=0}^\infty \frac{(q^{-n};q)_k\,(0;q)_k}{(aq;q)_k\,(q;q)_k}(-q x)^k\\
	&=\sum_{k=0}^{\infty}\prod_{j=1}^k \frac{q^{j-n}-q}{(1-aq^j)(1-q^j)} x^k=\sum_{k=0}^{n}\prod_{j=1}^k \frac{q^{j-n}-q}{(1-aq^j)(1-q^j)} x^k, \quad q\in (0,1),\quad a\in [0,1].
\end{align*}
We chose negative arguments $-x$ so that the coefficients
$$
a_{k:n}
=\prod_{j=1}^k\frac{q^{j-n}-q}{(1-aq^j)(1-q^j)}\ge 0
$$
are nonnegative for $a\le1< q^{-1}$. We are interested in a regime when $q\to 1-$. It turns out that a non-trivial profile is obtained if we let $q:=\eee^{-\lambda/n}$ for some fixed $\lambda>0$. Then the exponential profile is, for $\alpha \in (0,1)$,
\begin{align*}
	g(\alpha)
	=\lim_{n\to\infty}
	\frac{1}{n}\log a_{\lfloor \alpha n\rfloor:n}
	&=\lim_{n\to\infty}
	\frac{1}{n}\sum_{j=1}^{\lfloor \alpha n\rfloor}
	\log\left(
	\frac{\eee^{-\lambda(j/n-1)}-\eee^{-\lambda/n}}{(1-a\eee^{-\lambda j/n})(1-\eee^{-\lambda j/n})}
	\right)\\
	&=\lim_{n\to\infty}
	\frac{1}{n}\sum_{j=1}^{\lfloor \alpha n\rfloor}
	\log\left(
	\frac{\eee^{-\lambda(j/n-1)}-1}{(1-a\eee^{-\lambda j/n})(1-\eee^{-\lambda j/n})}
	\right)\\
	&=\int_0^\alpha
	\log\left(
	\frac{\eee^{\lambda(1-x)}-1}{(1-a\eee^{-\lambda x})(1-\eee^{-\lambda x})}
	\right){\rm d}x,
\end{align*}
by a Riemann summation in the last step. By Theorem \ref{theo:exp_profile_implies_zeros}, we have weak convergence of the empirical zero distributions
$$
\lsem P_n(-x;a|\eee^{-\lambda/n})\rsem_n\toweak \mu_{a|\lambda}(-\cdot),
$$ where $\mu_{a|\lambda}(-\cdot)$ has a Cauchy transform $G^{-}_{a|\lambda}$ such that $t\mapsto tG^{-}_{a|\lambda}(t)$ is the inverse of $\alpha\mapsto \eee ^{-g'(\alpha)}$. Differentiating and exponentiating $g$ yields
$$
\eee^{-g'(\alpha)}=	\frac{(1-a\eee^{-\lambda \alpha})(1-\eee^{-\lambda \alpha})}{\eee^\lambda \eee^{-\lambda\alpha}-1}
= f(h(\alpha)),
$$
where
$$
h(\alpha)=\eee^{-\lambda\alpha},\qquad
f(x)=\frac{(1-ax)(1-x)}{\eee^{\lambda}x-1}.
$$
Observe that $h^{\leftarrow}(x)=-\frac 1 \lambda\log x$. Solving a quadratic equation we find the inverse of $f$ as
$$
f^{\leftarrow}(t)
=\frac{1}{2a}\left(a+1+\eee^{\lambda }t\pm\sqrt{(a+1+\eee^{\lambda }t)^2-4a -4a t}\right).
$$
Hence, $tG^{-}_{a|\lambda}(t)=h^{\leftarrow}(f^{\leftarrow}(t))$ implies that the Cauchy transform $G_{a|\lambda}(t)=-G^{-}_{a|\lambda}(-t)$ of $\mu_{a|\lambda}$ is given by
\begin{align}\label{eq:q-G}
	G_{a|\lambda}(t)=-\frac 1 {\lambda t} \log \left[ \frac{1}{2a}\left(a+1-\eee^{\lambda }t +\sqrt{(a+1-\eee^{\lambda }t)^2+4a(t-1)}\right)\right],
\end{align}
where the branch of the square root is chosen such that $G_{a|\lambda}(t)\sim \frac{1}{t}$, as $t\to-\infty$. In order to find the support of $\mu_{a|\lambda}$ observe that the right-hand side of~\eqref{eq:q-G} has a branch cut on the positive real line either if
\begin{equation}\label{eq:q-G-quadratic_equation}
(a+1-\eee^{\lambda }t)^2+4a(t-1)\le 0
\end{equation}
or when the argument of the logarithm is negative. By solving the quadratic inequality we find that~\eqref{eq:q-G-quadratic_equation} holds if and only if $t\in [t_-,t_+]$, where
$$
t_\pm=\eee^{-\lambda}\big( a+1-2a\eee^{-\lambda}\pm 2\sqrt{a(1-a\eee^{-\lambda})(1-\eee^{-\lambda})}\big)=\eee^{-\lambda}\big(\sqrt{1-a\eee^{-\lambda}}\pm\sqrt{a(1-\eee^{-\lambda})}\big)^2.
$$
Observe that $0\leq t_{-}\leq t_{+}$ and, by the AM-GM inequality $t_{+}\leq 4\eee^{-\lambda}\sqrt{a(1-\eee^{-\lambda})(1-a\eee^{-\lambda})}$. The right-hand side, as a function of $a\in [0,1]$, attains the maximum at $a^{\ast}=\eee^{\lambda}/2$, if $\eee^{\lambda}\leq 2$ or at $a^{\ast}=1$, otherwise. In both cases, this yields $t_{+}\leq 1$. Moreover, for $t\in\mathbb{R}\setminus [t_-,t_+]$, the argument of the logarithm in~\eqref{eq:q-G} is nonpositive if and only if $t\geq \eee^{-\lambda}(a+1)$ and $t\leq 1$. Using also that
$t_{-}\leq \eee^{-\lambda}(a+1)$, we conclude that the natural domain of analyticity of $G_{a|\lambda}$ is
$$
\begin{cases}
\mathbb{C}\bsl [t_{-},t_{+}],& \text{if } \lambda\leq \log (1+a),\\
\mathbb{C}\bsl [t_{-},1],& \text{if } \lambda> \log (1+a).
\end{cases}
$$
\begin{figure}[t]
	\centering
	\includegraphics[width=0.33\linewidth]{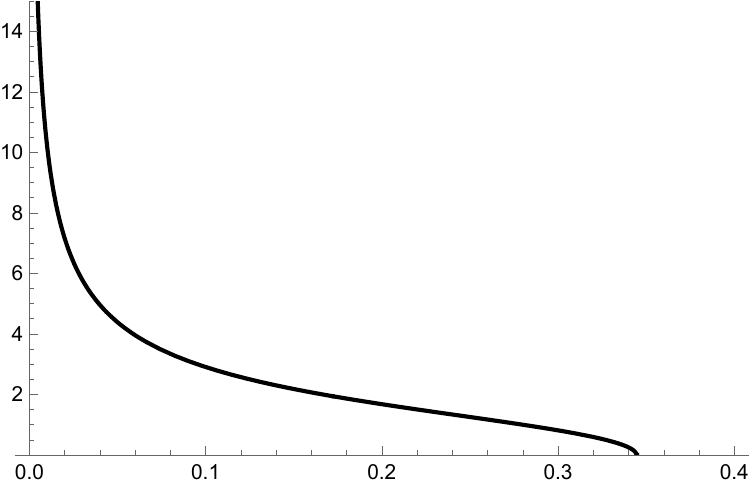}\includegraphics[width=0.33\linewidth]{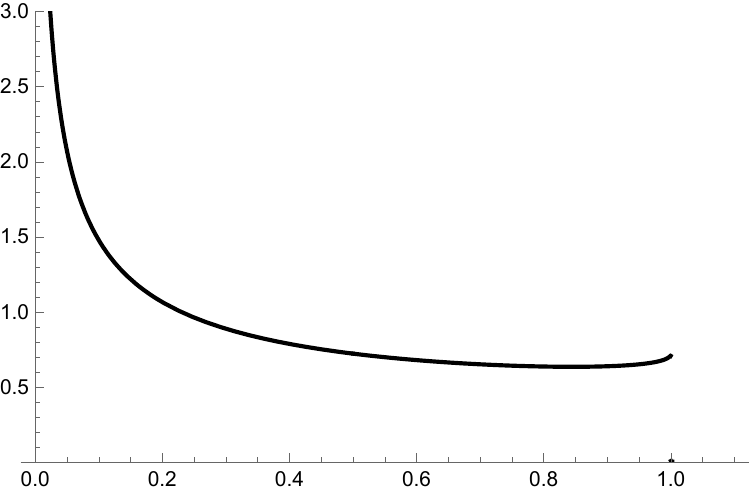}\includegraphics[width=0.33\linewidth]{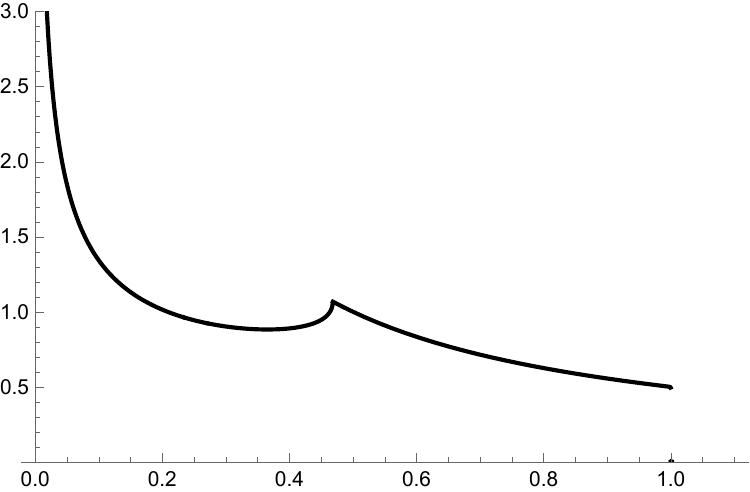} \includegraphics[width=0.33\linewidth]{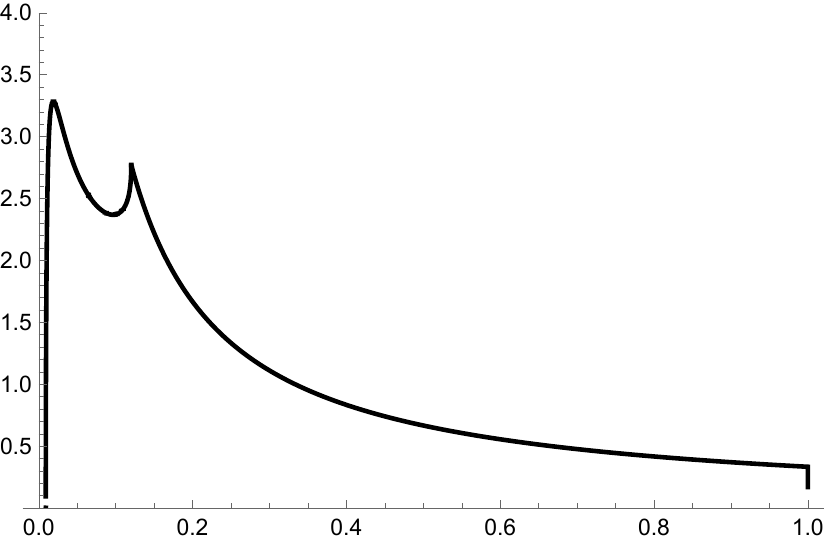}\includegraphics[width=0.33\linewidth]{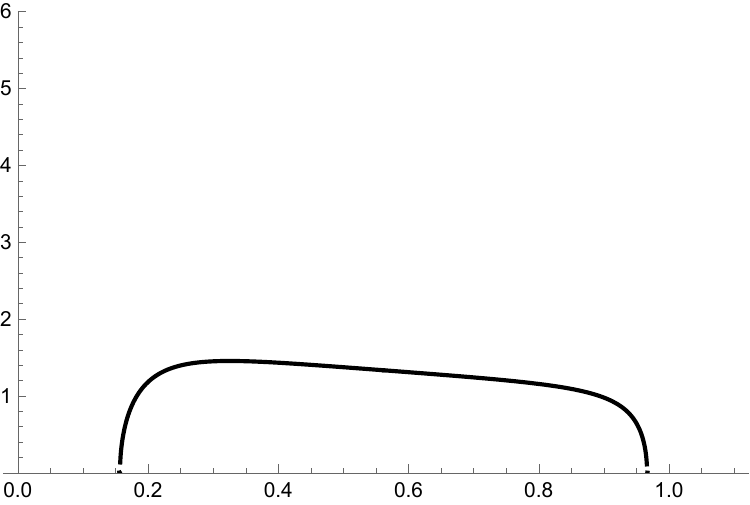}\includegraphics[width=0.33\linewidth]{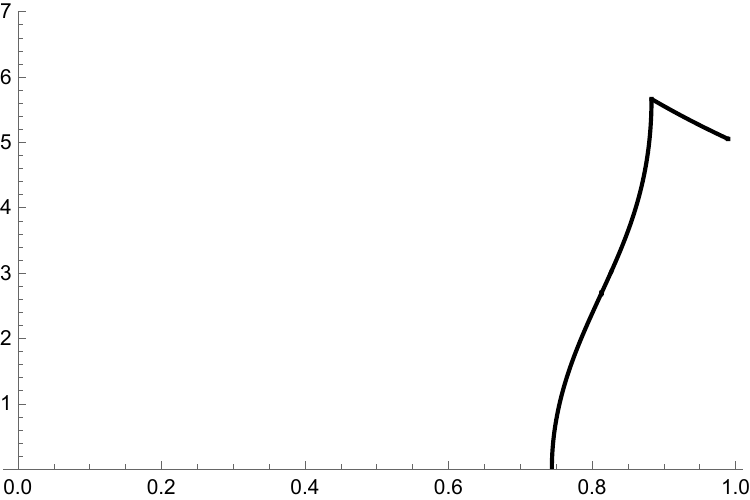}
	\caption{The limiting probability densities of zeros of little $q$-Laguerre polynomials for $a=1$ in the top row with $\lambda\in\{1/10,\log(2),2\}$ (from left to right) and in the bottom row for $a=1/3,\lambda=3$ (left), $a=1/2,\lambda=1/4$ (center) and $a=1/100,\lambda=1/5$ (right).}
	\label{fig:q-Laguerre}
\end{figure}
By the inversion formula~\eqref{eq:Perron_inversion} we conclude that the density $	p_{a|\lambda}(x)=-\frac 1  \pi \lim_{y\to 0} \Im G_{a|\lambda}(x+\ii y)$ of $\mu$ on $[0,+\infty)$ is
\begin{align}\label{eq:p_qa}
	p_{a|\lambda}(x) =\begin{cases}
		\frac 1 {\lambda x}, \ &\text{ if }x \in[t_+,1],\quad \lambda>\log(1+a), \\
		\frac 1 {\lambda \pi x} \arg\big(a+1-\eee^{\lambda }x+\eee^{\lambda}\ii\sqrt{(x-t_-)(t_+-x)}\big), \ &\text{ if }x\in[t_-,t_+], \\
		0,\ &\text{ otherwise.}
	\end{cases}
\end{align}
Here an additional term appearing when $\lambda>\log(1+a)$ is due to the logarithm crossing the branch cut. In this case $\log(w)=\log|w|+i(\arg(w)\pm\pi \epsilon)$ and $\epsilon\in\{0,1\}$ depends on the approaching direction.

Let us now discuss some illustrative simplifications. Set $a=1$, then $t_-=0$ and consider the symmetrized distribution that is obtained under the push-forward of $x\mapsto \pm \sqrt {x}$ (similar to the link \eqref{eq:Hermite-Laguerre}). The corresponding density is
$$\tilde p_{\lambda}(x)=\frac 1 {\lambda |x|}\Big(\frac 1 \pi \arg\big(2-\eee^{\lambda }x^2+\eee^{\lambda}\ii x\sqrt{t_+-x^2}\big)\ind_{[0,t_+]}(x^2)+\ind_{[t_+,1]}(x^2)\Big) $$
where $t_+=4\eee^{-\lambda}(1-\eee^{-\lambda})$. Interestingly, this is precisely the density of $q$-deformed GUE matrices, that has been uncovered in \cite{q-def}. Therein, discrete $q$-Hermite I polynomials have been investigated, which share the same exponential profile as $P_n(-x^2;1|q)$ and, hence, the same limiting empirical zero distribution.

Moreover, we can recover the continuous analogue in the $\lambda\to 0$ limit. Indeed, zooming into $\mu_{1|\lambda}$ as $\hat{\mu}_{1|\lambda}=\mu_{1|\lambda}(\lambda \cdot)$, the corresponding Cauchy transform satisfies $\hat{G}_{1|\lambda}(t)=\lambda G_{1|\lambda}(\lambda t)$. Using standard approximations like $\eee^\lambda=1+\lambda+o(\lambda)$, $\log(1+w)=w+o(w)$, we obtain from \eqref{eq:q-G}
\begin{align*}
	\lim_{\lambda\to 0+}\hat{G}_{1|\lambda} (t)  =-\lim_{\lambda\to 0}\frac 1 {\lambda t} \log \Big[ \frac{1}{2}\left(2-\eee^{\lambda }\lambda t+\sqrt{(2-\eee^{\lambda }\lambda t)^2+4(\lambda t-1)}\right)\Big]=\frac{1}{2t}\big(t-\sqrt{t(t-4)}\big).
\end{align*}
The right-hand side is the Cauchy transform of the standard Marchenko--Pastur law $\mathsf{MP}_{1,1}$, see~\eqref{eq:marchenko_pastur_Stieltjes}.

\section{Polynomials related to free probability} \label{sec:free_probab_ensembles}

\subsection{Polynomials related to \texorpdfstring{$\boxtimes$}{boxtimes}-infinitely divisible distributions}
For any two probability measures $\nu_1$ and $\nu_2$ on $[0,\infty)$, their multiplicative free convolution $\nu_1\boxtimes \nu_2$ is another probability measure on $[0,\infty)$ defined by Voiculescu in~\cite{voiculescu_symmetries,voiculescu_multiplication}. The class of probability measures that are infinitely divisible with respect to $\boxtimes$ has been described by Bercovici and Voiculescu~\cite{bercovici_voiculescu_levy_hincin,bercovici_voiculescu} in terms of a L\'evy--Khintchine-type representation. The basic building blocks of these $\boxtimes$-divisible distributions are the free multiplicative analogues of the normal and the Poisson distributions. In this section, we introduce polynomials which may be considered as finite free probability~\cite{marcus_spielman_srivastava,marcus2021polynomial} analogues of these basic $\boxtimes$-divisible distributions in the same way as Hermite (respectively, Laguerre) polynomials are finite free analogues of the free normal (semicircle) and free Poisson (Marchenko--Pastur) distributions.  In particular, the empirical zero distributions of the polynomials we introduce  converge to $\boxtimes$-divisible probability measures.

Let us first briefly  recall some notions from free probability; see, for example, \cite[\S~3.6]{voiculescu_nica_dykema_book}, \cite{BB05_semigroup,bercovici_voiculescu}.
Let $\mathcal M_{[0,\infty)}$ be the set of probability measures on $[0,\infty)$. The \textit{$\psi$-transform} of a probability measure $\nu\in \mathcal M_{[0,\infty)}$ is the function $\psi_\nu(z)$ defined on the open unit disk $\bD= \{z\in \C : |z|<1\}$ by
\begin{equation}\label{eq:psi_transf_def}
\psi_\nu(z) = \int_{[0,\infty)} \frac{uz}{1-uz} \nu(\dd u),
\qquad z\in \bD.
\end{equation}
The \textit{$S$-transform} and the \textit{$\Sigma$-transform} of $\nu\neq \delta_0$ are defined by
\begin{equation}\label{eq:S_transf_def}
S_\nu(z) = \frac {1+z}{z} \psi^{-1}_\nu(z),
\qquad
\Sigma_\nu(z) = S_\nu \left(\frac {z}{1-z}\right)
=
\frac 1z \psi^{-1}_\nu\left(\frac{z}{1-z}\right),
\end{equation}
respectively, if $|z|$ is sufficiently small. The free multiplicative convolution of $\nu_1,\nu_2 \in \mathcal M_{[0,\infty)}\backslash\{\delta_0\}$ can be characterized by
\begin{equation}\label{eq:S_transform_linearizes_conv}
S_{\nu_1 \boxtimes \nu_2} (z) = S_{\nu_1}(z)  S_{\nu_2} (z),
\qquad
\Sigma_{\nu_1 \boxtimes \nu_2} (z) = \Sigma_{\nu_1}(z)  \Sigma_{\nu_2} (z).
\end{equation}

The class of $\boxtimes$-infinitely divisible distributions has been described in~\cite{bercovici_voiculescu_levy_hincin,bercovici_voiculescu} as follows; see also~\cite[Theorem 2.3]{Arizmendi_Hasebe}.  A probability measure $\nu\neq \delta_0$ on $[0,\infty)$ is $\boxtimes$-infinitely divisible if and only if its $\Sigma$-transform admits a L\'evy--Khintchine-type representation
\begin{equation}\label{eq:free_mult_levy_khinchine}
\Sigma_\nu(z) = \exp(v_\nu(z)),
\quad
v_\nu(z) = c_0 + \int_{[0,+\infty]} \frac{1+tz}{z-t} \rho(\dd t),
\quad
z\in(-\infty,0),
\end{equation}
for some finite measure $\rho$ on $[0,+\infty]$ and some $c_0\in \R$. If $\rho$ has an atom at $+\infty$, we interpret $\frac{1+tz}{z-t}$ at $t=+\infty$ as $-z$.

In finite free probability~\cite{marcus2021polynomial,marcus_spielman_srivastava}, the analogue of $\boxtimes$ is an operation $\boxtimes_n$ defined on polynomials of degree at most $n$ as follows:
$$
\sum_{k=0}^n a_k' x^k \boxtimes_n \sum_{k=0}^n a_k'' x^k =  \sum_{k=0}^n (-1)^{n-k} \frac{a_k' a_k''}{\binom n k} x^k.
$$

Next we state  a general result allowing to construct sequences of polynomials whose empirical zero distributions converge.
\begin{proposition}\label{prop:general_conv_to_free_mult_ID_law}
For every $n\in \N$ let  $Q_n(x)=\sum_{k=0}^n (-1)^{n-k} \binom n k b_{k:n} x^k$ be a polynomial with real nonnegative roots and $b_{0:n}\geq 0,\ldots, b_{n:n}\geq 0$. Suppose that for every $\alpha \in (0,1)$, we have
$$
\lim_{n\to\infty} b_{\lfloor \alpha n\rfloor:n} = \eee^{\tilde g (\alpha)}
$$
with some function $\tilde g:(0,1) \to \R$. Then, the polynomials
$$
P_n(x):= \underbrace{Q_n(x) \boxtimes_n  \cdots \boxtimes_n Q_n(x)}_{n\text{ times}} = \sum_{k=0}^n (-1)^{n-k} \binom n k  b^n_{k:n} x^k
$$
also have real nonnegative roots and $\lsem P_n\rsem_n$ converges weakly on $[0,\infty)$ to a probability measure $\nu$ with
$$
S_\nu(t) = \eee^{\tilde g'(1+t)},
\quad
t\in (-1,0)
\qquad
\Sigma_\nu(z) = \eee^{\tilde g'\left(\frac 1 {1-z}\right)},
\quad
z\in (-\infty, 0).
$$
\end{proposition}
\begin{proof}
The fact that  all roots of $P_n$ are nonnegative  follows from a theorem of Szeg\H{o}, see~\cite[Theorem~1.6]{marcus_spielman_srivastava}.
The sequence $((-1)^n P_n(-x))_{n\in \N}$ has exponential profile
$$
g(\alpha) := \lim_{n\to\infty} \frac 1n \log \left(  \binom n {[\alpha n]} b_{[\alpha n]:n}^n\right) =  \tilde g (\alpha) -\alpha\log\alpha-(1-\alpha)\log(1-\alpha), \qquad \alpha \in (0,1).
$$
Recall that the profile is automatically infinitely differentiable. Its derivative is $g'(\alpha)=\log(\frac{1-\alpha}{\alpha})+\tilde g'(\alpha)$.  By Theorem~\ref{theo:exp_profile_implies_zeros}, $\lsem P_n(-x)\rsem_n$ converges weakly to some probability measure $\mu$ on $[-\infty, 0]$ such that $\alpha\mapsto \eee^{-g'(\alpha)}$ is the inverse of $t\mapsto tG_{\mu}(t)$, where $G_{\mu}$ is the Cauchy transform of $\mu$. Moreover, we have $\mu(\{0\})= \mu(\{-\infty\}) = 0$ since the profile is defined on the whole of $(0,1)$. Let $\mu_{+}$ be a reflection of $\mu$ with respect to $0$ given by $\mu_{+}(A):=\mu(-A)$, for Borel sets $A\subseteq \R$. Now,  Lemma 6.1 in~\cite{jalowy_kabluchko_marynych_zeros_profiles_part_I} states that the $S$-transform $S_{\mu_{+}}$ of $\mu_{+}$ is related to the profile $g$ by
\begin{equation}\label{eq:S_mu-via-profiles}
S_{\mu_{+}}(t)=-\frac{1+t}{t}\eee^{g'(1+t)}, \qquad t\in (-1,0).
\end{equation}
Thus, $\lsem P_n \rsem_n$ converges weakly on $[0,\infty)$ to the probability measure $\nu:= \mu_+$ with the $S$- and $\Sigma$-transforms
$$
S_\nu(t) = - \frac{1+t}{t}\eee^{g'(1+t)}=\eee^{\tilde g '(1+t)}, 
\quad
\Sigma_\nu(z) = S_\nu\left(\frac {z}{1-z}\right) = \eee^{\tilde g'\left(\frac 1 {1-z}\right)}, 
$$
for all $t\in (-1,0)$ and $z\in (-\infty, 0)$.
Note that $\nu$ is determined uniquely by its $S$- and $\Sigma$-transforms on these intervals since these transforms are analytic inside some disk centered at $0$.
\end{proof}

\begin{remark}
We are mostly interested in the case when $\nu$ is a $\boxtimes$-infinitely divisible distribution on $[0,\infty)$ with representation~\eqref{eq:free_mult_levy_khinchine}. Then, $\tilde g'(\alpha)=v_\nu(1-1/\alpha)$.
\end{remark}

In the following, we shall construct  explicit polynomials related to the free multiplicative normal and Poisson distributions.  These distributions are $\boxtimes$-infinitely divisible with $\rho = a \delta_1$, $a>0$ (in the normal case) and $\rho = a \delta_t$, $t\in [0,+\infty] \backslash\{1\}$, $a>0$ (in the Poisson case).

\subsubsection{Free multiplicative Hermite polynomials}
For $\sigma>0$ consider the polynomials
$$
G_n(x; \sigma^2) = \sum_{k=0}^n  (-1)^{n-k} \binom {n}{k} \eee^{\frac {\sigma^2 k(n-k)}{2}}  x^k,
\qquad n\in\mathbb{N}.
$$
\begin{theorem}\label{the0:free_mult_hermite}
For every $\sigma>0$ and $n\in \N$, all zeros of $G_n(x; \sigma^2)$ are real and positive. The empirical distribution of zeros of $G_n(x; \sigma^2/n)$ converges, as $n\to\infty$, to a probability measure $\nu_{\sigma^2}$ on $[0,\infty)$ with
$$
S_{\nu_{\sigma^2}}(t) = \eee^{-\sigma^2 (t + \frac 12)},
\quad
t\in (-1,0),
\qquad
\Sigma_{\nu_{\sigma^2}}(z) = \eee^{\frac{\sigma^2}{2} \frac{z+1}{z-1}},
\quad
z\in (-\infty, 0).
$$
\end{theorem}
\begin{remark}
The distribution $\nu_{\sigma^2}$ is called the free multiplicative normal distribution with parameter $\sigma^2$. It appeared  in~\cite[Lemma 7.1]{bercovici_voiculescu_levy_hincin} and is $\boxtimes$-infinitely divisible with $\rho= \frac 12 \sigma^2 \delta_1$ and $c_0 = 0$ in~\eqref{eq:free_mult_levy_khinchine}. The moments and the density of $\nu_{\sigma^2}$ have been determined by Biane~\cite{biane,biane_segal_bargmann}.  For more information, see~\cite{zhong_free_normal}, \cite[Sections~4,5]{zhong_free_brownian}, \cite{anshelevich_wang_zhong,demni2011spectraldistributionfreeunitary,demni2016spectraldistributionlargesizebrownian}.
\end{remark}
\begin{remark}
For a fixed $n\in \N$, the polynomials $G_n(x; \sigma^2)$, $\sigma^2>0$, form a semigroup with respect to $\boxtimes_n$ in the sense that $G_n(x; \sigma_1^2) \boxtimes_n G_n(x; \sigma_2^2) = G_n(x; \sigma_1^2+ \sigma_2^2)$ for all $\sigma_1^2, \sigma_2^2 >0$.
\end{remark}

We need some preparations before we can start with the proof of Theorem~\ref{the0:free_mult_hermite}.  An entire function $f(z)$ belongs to the Laguerre--P\'{o}lya class if it possesses a representation as a Hadamard product
\begin{equation}\label{eq:polya_laguerre_def}
f(z)=Cz^m\eee^{cz-\frac{\sigma^2}{2}z^2}\prod_{j=1}^N\Big(1-\frac{z}{x_j}\Big)\eee^{\frac z {x_j}}
\end{equation}
for some $C,c\in\R$, $\sigma^2\geq 0$, $m\in\N\cup\{0\}$,  $N\in\N\cup\{\infty\}$ and $x_j\in\R\bsl\{0\}$ satisfying $\sum_{j=1}^Nx_j^{-2}<\infty$.  The following theorem of Laguerre  can be found in~\cite[Section~5.7 and Theorem~5.6.12]{rahman_schmeisser_book_polys} or~\cite[25:Theorem]{warner_book_zeros}.
\begin{theorem}[Laguerre]\label{theo:laguerre_polys}
If $f(z)$ is a function from the Laguerre--P\'olya class  with only negative zeros and if all zeros of the polynomial $p(x) = \sum_{k=0}^n c_k x^k \in \R[x]$ are real, then all zeros of the polynomial $\sum_{k=0}^n f(k) c_k x^k$ are also real.
\end{theorem}
\begin{remark}\label{rem:laguerre_polys_special_case}
Since all zeros of the  polynomial $p(x) = (x-1)^{n} = \sum_{k=0}^n  (-1)^{n-k} \binom {n}{k} x^k$ are real, Laguerre's theorem implies that  $\sum_{k=0}^n (-1)^{n-k} f(k) \binom nk x^k$ is real-rooted.
\end{remark}

\begin{proof}[Proof of Theorem~\ref{the0:free_mult_hermite}]
The function $f(z) = \eee^{\frac 12 {\sigma^2 z(n-z)}}$ belongs to the Laguerre--P\'olya class by~\eqref{eq:polya_laguerre_def}. Remark~\ref{rem:laguerre_polys_special_case} implies that all zeros of $G_n(x; \sigma^2)= \sum_{k=0}^n (-1)^{n-k} f(k) \binom nk x^k$ are real. Since the coefficients of $G_n(x;\sigma^2)$ alternate signs and $G_n(0; \sigma^2) \neq 0$, all roots are in fact positive.

To identify the weak limit of $\lsem G_n(x; \sigma^2/n)\rsem_n$, we write $G_n(x; \sigma^2/n)  =\sum_{k=0}^n (-1)^{n-k} \binom {n}{k} b_{k:n}^n x^k$  with $b_{k:n} = \eee^{\frac {\sigma^2} 2 \frac kn (1-\frac kn)}$ and observe that $b_{\lfloor \alpha n\rfloor:n} \to \eee^{\frac {\sigma^2} 2 \alpha (1-\alpha)}$ for every $\alpha\in (0,1)$. We can therefore apply Proposition~\ref{prop:general_conv_to_free_mult_ID_law} with $\tilde g(\alpha) = \frac {\sigma^2} 2 \alpha (1-\alpha)$ and $\tilde g'(\alpha) = \sigma^2 (\frac 12- \alpha)$. It follows that $\lsem G_n(x; \sigma^2/n)\rsem_n$ converges weakly to a probability measure $\nu = \nu_{\sigma^2}$ with
$$
S_\nu(t) = \eee^{\tilde g'(1+t)} =  \eee^{-\sigma^2 (t + \frac 12)},
\qquad
\Sigma_{\nu}(z) = \eee^{\tilde g'\left(\frac 1 {1-z}\right)} = \eee^{\sigma^2 \left(\frac 12 - \frac 1 {1-z}\right)} = \eee^{\frac{\sigma^2}{2} \, \frac{z+1}{z-1}},
\qquad
$$
for all $t\in (-1,0)$ and $z\in (-\infty, 0)$.
\end{proof}

\begin{remark}
The polynomials $G_n(x; \sigma^2)$ are similar to unitary Hermite polynomials~\cite{kabluchko2024leeyangzeroescurieweissferromagnet,mirabelli2021hermitian} defined by
$$
H_n(x; \sigma^2) := \sum_{k=0}^{n}(-1)^{n-k}\binom{n}{k}\eee^{-\frac{\sigma^2 k(n-k)}{2}}x^k,
$$
with the only difference being the sign in the exponent. It is known that all zeros of $H_n(x; \sigma^2)$ lie on the unit circle, see Lemma 2.1 in~\cite{kabluchko2024leeyangzeroescurieweissferromagnet}. The empirical distribution of zeros of $H_n(x; \sigma^2/n)$ converges, as $n\to\infty$, to the free \emph{unitary} normal distribution on the unit circle with parameter $\sigma^2$, see Theorem~2.2 in~\cite{kabluchko2024leeyangzeroescurieweissferromagnet}.
\end{remark}

\subsubsection{Polynomials related to free multiplicative Poisson distributions}
For $n\in\mathbb{N}$ and parameters $b \in \R \backslash(-n,0)$, $c\geq 0$ we consider the polynomial
$$
P_{n}(x; b, c) = \sum_{k=0}^n  (-1)^{n-k} \binom {n}{k} |k+b|^{c} x^k.
$$

\begin{lemma}\label{lem:poly_poisson_multipl_free_real_rooted}
Let $c\in \N_0$ and either $b\geq 0$ or $b \leq  -n$. Then, all roots of $P_{n}(x; b, c)$ are real.
\end{lemma}
\begin{proof}
We apply Theorem~\ref{theo:laguerre_polys} and Remark~\ref{rem:laguerre_polys_special_case}.
Let first $b\geq 0$. Then, for all $k\in \{0,\ldots, n\}$ we have  $|k+b|^{c} = (k+b)^c$. Since $f(z) := (z+b)^c$ is a polynomial with nonpositive roots (in particular, $f(z)$ belongs to the Laguerre--P\'olya class), Laguerre's theorem applies and proves that $P_{n}(x; b,c)$ is real-rooted.  Let now $b\leq -n$. Then, for all $k\in \{0,\ldots, n\}$ we have $|k+b|^{c} = (-b-k)^c$. Since $f(z) := (-b-z)^c$ is a polynomial with nonpositive roots, Laguerre's theorem implies that $P_{n}(x; b,c)$ is real-rooted. A different proof of the lemma will be given in Remark~\ref{rem:mult_laguerre_motivation}.
\end{proof}

\begin{remark}
The assumption that $c$ is integer is essential. For example, the polynomial $P_{3}(x; 0,0.5)$ is not real-rooted. On the other hand, numerical simulations support the following conjecture. The polynomial $P_{n}(x; b, c)$ is real-rooted for every $n\in \N$, $b\in \R \bsl (-n,0)$ and every real $c\geq n$.
\end{remark}

\begin{remark}\label{rem:free_poisson_polys_semigroup}
The polynomials $P_{n}(x; b,c)$, $c\geq 0$, form a semigroup with respect to $\boxtimes_n$ in the sense that $P_{n}(x; b,c') \boxtimes_n P_{n}(x; b,c'') = P_{n}(x; b,c'+c'')$ for all $c', c''\geq 0$.
\end{remark}

\begin{theorem}\label{theo:poly_free_mult_poisson_empirical}
Fix $\beta \in \R\backslash(-1,0)$ and $\gamma > 0$. Let $(b_n)_{n\in \N}\subseteq \R$ and $(c_n)_{n\in \N}\subseteq \N_0$ be sequences such that $b_n/n\to \beta$ and $c_n/n \to \gamma$ as $n\to\infty$. Let also $b_n\in \R \backslash (-n,0)$ for every $n\in \N$. Then, $\lsem P_{n}(x; b_n, c_n)\rsem_n$ converges weakly to a probability measure $\nu = \nu_{\beta, \gamma}$ on $[0,\infty)$ with
$$
S_{\nu}(t) = \eee^{\frac{\gamma}{t + \beta+1}},
\quad
t\in (-1,0),
\qquad
\Sigma_{\nu}(z) = \eee^{\frac{\gamma (1-z)}{1 + \beta (1-z)}},
\quad
z\in (-\infty, 0).
$$
\end{theorem}
\begin{proof}
Write $n^{-c_n} P_n(x; b_n, c_n) = \sum_{k=0}^n (-1)^{n-k} \binom nk b_{k:n}^n x^k$ with $b_{k:n} = |(k+b_n)/n|^{c_n/n}$. Note that $b_{\lfloor \alpha n\rfloor:n} \to |\alpha + \beta|^\gamma$ as $n\to\infty$, for all $\alpha \in (0,1)$. We can apply Proposition~\ref{prop:general_conv_to_free_mult_ID_law} with $\tilde g(\alpha) = \gamma \log |\alpha + \beta|$. If $\beta \geq 0$, then $\tilde g(\alpha) = \gamma \log (\alpha + \beta)$, whereas  for $\beta \leq -1$ we have $\tilde g(\alpha) = \gamma \log (-\alpha - \beta)$. In both cases, $\tilde g'(\alpha) = \gamma/(\alpha+\beta)$.  It follows from Proposition~\ref{prop:general_conv_to_free_mult_ID_law} that $\lsem P_n(x; b_n,c_n)\rsem_n$ converges weakly to a probability measure $\nu = \nu_{\beta, \gamma}$ with
$$
S_{\nu}(t) = \eee^{\tilde g'(1+t)} = \eee^{\frac{\gamma}{t + \beta+1}} ,
\qquad
\Sigma_{\nu}(z) = \eee^{\tilde g'\left(\frac 1 {1-z}\right)} = \eee^{\frac {\gamma}{\frac 1 {1-z} + \beta}} = \eee^{\frac{\gamma (1-z)}{1 + \beta (1-z)}},
\qquad
$$
for all $t\in (-1,0)$ and $z\in (-\infty, 0)$.
\end{proof}

The distribution $\nu_{\beta, \gamma}$ appeared in~\cite[Lemma 7.2]{bercovici_voiculescu_levy_hincin}.  It is the free mutiplicative analogue of the Poisson distribution. If $\beta\neq 0$, then it is $\boxtimes$-infinitely divisible with $\rho= \frac{\gamma}{\beta}\frac{t_0-1}{t_0^2+1}\delta_{t_0}$ and $c_0 = \frac{\gamma}{\beta}\frac{t_0+1}{t_0^2+1}$ in~\eqref{eq:free_mult_levy_khinchine}, where $t_0=\beta^{-1}+1$. If $\beta=0$, it is $\boxtimes$-infinitely divisible as well with $c_0=\gamma$ and $\rho=\gamma\delta_{+\infty}$.

\begin{remark}\label{rem:mult_laguerre_motivation}
The polynomials $P_n(x; b,c)$ can be considered as a finite free \emph{multiplicative} analogue of the Poisson distribution in the same way as the Laguerre polynomials are the finite free \emph{additive} analogue of the Poisson distribution~\cite{marcus2021polynomial}. Let us provide motivation for this.
Let $U_n$ be a random vector with uniform distribution on the unit sphere in $\R^n$. Consider the linear operator ${\bf A}_n: \R^n \to \R^n$ given by ${\bf A}_n:=\mathbf{Id}_n+(\mu-1)U_nU_n^T$, which satisfies ${\bf A}_n v = v$ for $v\in U_n^\perp$ and ${\bf A}_n U_n = \mu U_n$, where $\mu\geq 0$, $\mu \neq 1$, and $U_n^\perp$ denotes the orthogonal complement of the one-dimensional subspace spanned by $U_n$.  The characteristic polynomial of ${\bf A}_n$ is deterministic and given by
\begin{align*}
\det ({\bf A}_{n}-\mathbf{Id}_n x)
&=
(1-x)^{n-1} (\mu-x) = \sum_{\ell\in \Z} (-1)^{\ell+1} \binom {n-1}{\ell} x^\ell (x-\mu)\\
& =  \sum_{k\in \Z}(-1)^{k} x^k \left(\binom{n-1}{k-1} + \mu \binom{n-1}{k}\right)=
\sum_{k\in \Z}(-1)^{k} x^k \binom{n}{k} \left( \mu + (1-\mu) \frac kn\right)\\
&=\frac{1-\mu}{n} (-1)^n ( \sgn b ) P_n(x; b, 1), \quad b:= \frac{\mu n}{1-\mu}  \in \R \bsl [-n,0).
\end{align*}
It follows that $P_n(x; b, 1)$ is real-rooted, which by Remark~\ref{rem:free_poisson_polys_semigroup} and Szeg\H{o}'s theorem~\cite[Theorem~1.6]{marcus_spielman_srivastava} implies that $P_n(x; b, c)$ is real-rooted for every $c\in \N$. This gives another proof of Lemma~\ref{lem:poly_poisson_multipl_free_real_rooted} (assuming $b\neq -n$).

Note that the normalized spectral measure of ${\bf A}_n$ is $\frac{n-1}{n} \delta_1 + \frac 1n \delta_\mu$. The classical Poisson limit theorem implies that the classical $\lfloor \lambda n\rfloor$-th convolution power of $\frac{n-1}{n} \delta_0 + \frac 1n \delta_{m}$ converges, as $n\to\infty$,  to $m \cdot \xi$, where the random variable $\xi$ has Poisson distribution with parameter $\lambda>0$ and $m\neq 0$. For a finite free multiplicative  analogue of this result, let ${\bf A}_{1;n},\ldots, {\bf A}_{c; n}$ be $c$ independent realizations of ${\bf A}_n$, where $c\in \N$, and note that their distribution is invariant under conjugation with a Haar-unitary matrix. By~\cite[Theorem~1.5]{marcus_spielman_srivastava}, the expected characteristic polynomial of the product ${\bf A}_{1;n}\ldots {\bf A}_{c;n}$ is
\begin{align*}
\E \det ({\bf A}_{1;n}\ldots {\bf A}_{c;n}-\mathbf{Id}_n x)
&=(-1)^n\E \det (\mathbf{Id}_n x-{\bf A}_{1;n}\ldots {\bf A}_{c;n})\\
&=(-1)^n
\underbrace{(\E \det (\mathbf{Id}_n x -{\bf A}_{n})\boxtimes_n\cdots\boxtimes_n (\E \det (\mathbf{Id}_n x-{\bf A}_{n}))}_{c\text{ times}}\\
&=
(-1)^{n}\left(\frac{1-\mu}{n} \sgn b\right)^c  \underbrace{P_n(x; b, 1) \boxtimes_n\cdots\boxtimes_n P_n(x;b,1)}_{c\text{ times}}\\
&= (-1)^{n}\left(\frac{1-\mu}{n}\sgn b \right)^c  P_n(x; b,c),
\end{align*}
where the last step follows from the semigroup property stated in Remark~\ref{rem:free_poisson_polys_semigroup}.  This explains why it is natural to expect the free multiplicative Poisson distribution  to appear as  a  limit in Theorem~\ref{theo:poly_free_mult_poisson_empirical}.
Another approach to a free multiplicative analogue of the Poisson limit theorem can be found in~\cite[Section~10.5]{arizmendi2024s}.   For polynomials that are similar to $P_n(x; b,c)$ but have zeros on the unit circle, see~\cite{kabluchko2024repeateddifferentiationfreeunitary}. Sidi polynomials studied in~\cite{lubinski_sidi_strong_asympt,lubinski_sidi_composite_biorthogonal,sidi_numerical_quad_nonlinear_seq,sidi_num_quad_infinite_range,sidi_lubinski_on_the_zeros} are also closely related to $P_n(x; b,c)$.
\end{remark}

\subsection{Special polynomials generated by differential operators of hypergeometric type}
In this and the next subsection, we explore two distinct generalizations of the formula ${\rm He}_n(x)=\eee^{-\frac{1}{2}D^2}x^n$ each yielding real-rooted polynomials that can be expressed as differential flows of monomials. First, we turn to differential operators of the form $\eee^{-cnH}$ for some $c>0$ and a word $H=\prod_s B_s$ of operators $B_s$, each of which either decreases the degree of a polynomial by one or preserves the degree. More precisely, fix a parameter $m\in\N$ describing the degree decrease of $H$. For $j+1\ge m$, $b_1,\dots,b_j\in [1,+\infty)$, $b_0=1$ and define a collection of operators
\begin{align*}
	B_s&:=\frac{D}{nm} +\frac {b_s-1}x,\quad 0\leq s<m,\\
	B_s&:=\frac{xD}{nm} + {b_s-1},\quad m\leq s\leq j.
\end{align*}
Thus, the operators $B_0=\frac{D}{nm},B_1,\ldots,B_{m-1}$ are degree decreasing and the operators $B_m,\ldots,B_{j}$ are degree preserving. The following formula holds true
\begin{equation}\label{eq:B_s_action}
	B_s x^{\ell}=\left(\frac{\ell}{nm}+b_s-1\right)x^{\ell-1+1_{\{s\geq m\}}},\quad s=0,\ldots,j,\quad \ell\geq 1.
\end{equation}
A somewhat unusual choice of scaling will become apparent from the result below. As we shall see, the order in which we apply the operators $B_s$, $s=0,1,\ldots,j$, is irrelevant in the limit, provided (for technical reasons) only that the first operator is the differentiation $B_0=D/{nm}$. Hence, we fix a permutation $\sigma$ of $\{1,2,\ldots,j\}$ and set
\begin{align}\label{eq:Hamiltonian}
	H=B_{\sigma(j)}\dots B_{\sigma(1)}B_0.
\end{align}
Using the notation $t_i:=\sum_{k=1}^{i}1_{\{\sigma(k)\geq m\}}$, $i=0,\ldots,m$, formula~\eqref{eq:B_s_action} implies
\begin{equation}\label{eq:H_action}
	Hx^{\ell}=\frac{\ell}{nm}\prod_{s=1}^{j}\left(\frac{\ell-1+t_{s-1}}{nm}+b_{\sigma(s)}-1\right)x^{\ell-m},\quad \ell\geq m.
\end{equation}
We shall show below that  under the stated notation and assumptions, for all sufficiently large $n$, we have
\begin{align}\label{eq:hamilton-flow}
	\eee^{-cnH}x^{nm}=(-c)^n\HyperP{{\emptyset}}{{\bf b^{(n)}}}{n}(-x^m/c)
\end{align}
for some appropriate $\mathbf b^{(n)}\in[1-1/n,+\infty)^{j}$ such that $\mathbf b^{(n)}\to (b_1,\dots,b_j)=:{\bf b}$, as $n\to\infty$. By Corollary~\ref{cor:real_zeros_hypergeo_polys_2} the roots of $\HyperP{{\emptyset}}{{\bf b^{(n)}}}{n}$ are nonpositive, hence the zeros of $\eee^{-cnH}x^{nm}$ are concentrated on the rays $\{r\eee^{2\pi i/m}:r\ge 0\}$ through the $m$th roots of unity. The function $S(x)=-x^{m}/c$ maps the union of these rays to the closed negative half-line. Thus,~\eqref{eq:hamilton-flow} implies
$$
\lsem \eee^{-cnH}x^{nm}\rsem_{nm}\circ S^{-1}=\lsem \HyperP{{\emptyset}}{{\bf b^{(n)}}}{n}(x)\rsem_n.
$$
Corollary~\ref{cor:real_zeros_hypergeo_polys_2} yields the following result.
\begin{theorem}\label{thm:hamilton-flow}
	Under the stated notation and assumptions, we have the convergence of the push-forwards measures
	\begin{align}
		\lsem \eee^{-cnH}x^{nm}\rsem_{nm}\circ S^{-1}\toweak \mu,
	\end{align}
	where $\mu$ is a probability measure on $(-\infty,0]$ with the Cauchy transform $G^{\varnothing,{\bf b}}$ such that the inverse of $t\mapsto tG^{\varnothing,{\bf b}}(t)$ is $\alpha\mapsto\frac {\alpha} {1-\alpha}\prod_{s=1}^j( b_s-1+\alpha)$.
\end{theorem}
\begin{proof}
	We only need to check~\eqref{eq:hamilton-flow}. The hypergeometric polynomial on the right hand side can be rewritten as
	\begin{align}\label{eq:hypergeo-poly-hamilton}
		(-c)^n\HyperP{\emptyset}{\mathbf b^{(n)}}{n}(-x^m/c)&= \sum_{k=0}^n\binom n k n^{-j(n-k)} {(\mathbf b^{(n)} n)^{\underline{n-k}}} (-c)^{n-k}	x^{mk}\nonumber\\
		&=\sum_{k=0}^n \frac {(-c)^kn^k} {k!} \prod_{s=0}^j \frac{\Gamma(b_s^{(n)}n+1)}{n^k\Gamma(b_s^{(n)}n-k+1)}	x^{m(n-k)}.
	\end{align}
	On the other hand
	$$
	\eee^{-cnH}x^{nm}=\sum_{k=0}^{\infty}\frac {(-c)^kn^k} {k!} H^k x^{nm}=\sum_{k=0}^{n}\frac {(-c)^kn^k} {k!} H^k x^{nm},
	$$
	because each application of $H$ reduces the degree of $x^{nm}$ by $m$. Hence $H^{n}x^{nm}$ is a constant, $B_0 H^{n}x^{nm}=0$ and, therefore, $H^{n+1}x^{nm}=0$. Thus, it remains to prove that
	\begin{align}\label{eq:hamilton-proof}
		H^k x^{nm}=\prod_{s=0}^j \frac{\Gamma(b_s^{(n)}n+1)}{n^k\Gamma(b_s^{(n)}n-k+1)}	x^{m(n-k)},\quad k=0,\ldots,n,
	\end{align}
	for some choice of $b_s^{(n)}$, $s=1,\ldots,j$ which depends on the permutation $\sigma$. Iterating formula~\eqref{eq:H_action} we conclude that
	\begin{align*}
		H^k x^{nm}&=\left(\prod_{\ell=0}^{k-1}\frac{m(n-\ell)}{nm}\right)\prod_{s=1}^{j}\prod_{\ell=0}^{k-1}\left(\frac{m(n-\ell)-1+t_{s-1}}{nm}+b_{\sigma(s)}-1\right) x^{m(n-k)}\\
		&=\left(\prod_{\ell=0}^{k-1}\frac{m(n-\ell)}{nm}\right)\prod_{s=1}^{j}\prod_{\ell=0}^{k-1}\left(\frac{m(n-\ell)-1+t_{\sigma^{-1}(s)-1}}{nm}+b_{s}-1\right) x^{m(n-k)}.
	\end{align*}
	The first product on the right-hand side is equal to $n^{-k}\Gamma(b_0^{(n)}n+1)/\Gamma(b_0^{(n)}n-k+1)$ if we set $b_0^{(n)}=1$ and matches the product with $s=0$ on the right-hand side of~\eqref{eq:hamilton-proof}. Put
	\begin{equation}\label{eq:b_n_choice}
		b_s^{(n)}:=b_{s}+\frac{t_{\sigma^{-1}(s)-1}-1}{mn},\quad s=1,\ldots,j,
	\end{equation}
	and observe that the corresponding product on the right-hand side of~\eqref{eq:hamilton-proof} is equal to
	$$
	\frac{\Gamma(b_s^{(n)}n+1)}{n^k\Gamma(b_s^{(n)}n-k+1)}	=\prod_{\ell=0}^{k-1}\frac{b_s^{(n)}n-\ell}{n}=\prod_{\ell=0}^{k-1}\left(\frac{m(n-\ell)-1+t_{\sigma^{-1}(s)-1}}{nm}+b_{s}-1\right).
	$$
	Thus,~\eqref{eq:hamilton-proof} holds true under the choice in~\eqref{eq:b_n_choice}. It remains to note that
	$$
	b_s^{(n)}:=b_{s}+\frac{t_{\sigma^{-1}(s)-1}-1}{mn}\geq b_s-\frac{1}{n}\geq 1-\frac{1}{n}\quad\text{and}\quad \lim_{n\to\infty}b_s^{(n)}=b,\quad s=1,\ldots,j.
	$$
\end{proof}

\begin{remark} Note that in the case $m=1,2$, it follows that $\eee^{-cnH}x^{nm}$ is real-rooted.
	One may extend the model to factors $(B_s)_{s=0,1,\dots,2j+1-m}=\Big(\frac D {nm}  ,\frac D {nm} +\frac {b_1-1}x,\dots,\frac D {nm} +\frac {b_j-1}x,x,\dots,x\Big)$, to reorder separated multiplication and degree decreasing operators, or even including integration operators, and the parameters $b_s$ could also mildly depend on $n$. However, in these extensions it is crucial to restrict the permutation $\sigma$ and hence we skip this extension for not overcomplicating the setting.
\end{remark}

We shall now discuss some particular cases of Theorem~\ref{thm:hamilton-flow}.

\vspace{2mm}
\noindent
{\sc Case I} in which $m=j=1$, $b_0=1$ and $b_1=b_1^{(n)}=1+(\gamma_n+1)/n\ge 1$ for a sequence of real numbers $(\gamma_n)$ such that $\lim_{n\to\infty} \frac 1n \gamma_n  = \gamma_*\geq 0$. Then,
$$
B_0=\frac Dn,\quad B_1=\frac {xD}{n}+ \frac{\gamma_n+1}{n}
$$
and we obtain a differential flow that maps monomials to suitably rescaled Laguerre polynomials
$$
\eee^{-cn(x\frac {D^2}{n^2}+(\gamma_n+1)\frac D{n^2})}x^n=(-c)^n\HyperP{{\emptyset}}{{1+\gamma_n/n}}{n}(-x/c)=\tfrac {(-c)^nn!}{n^n}L_n^{(\gamma_n)}(nx/c),
$$
where we have used~\eqref{eq:hamilton-flow} and~\eqref{eq:b_n_choice}. The weak convergence of the empirical zero distribution to a Marchenko--Pastur distribution has already been discussed in Corollary~\ref{cor:laguerre_polys_zeros}.

\vspace{2mm}
\noindent
{\sc Case II} in which $m=2$, $j=1$, $b_0=b_1=1$ and $c=1$. Thus,
$$
B_0=\frac{D}{2n},\quad B_1=\frac{D}{2n}
$$
and the resulting operator $\eee^{-cH}$ is the heat flow. Using~\eqref{eq:hamilton-flow} and~\eqref{eq:b_n_choice} we conclude that the even degree Hermite polynomials are
$$\eee^{-nH}x^{2n}=\eee^{-\frac 1 {4n} D^2}x^{2n}=(-1)^n\HyperP{{\emptyset}}{1-\frac 1 {2n}}{n}(-x^2)=\tfrac {n!}{(-n)^n}L_n^{(-\frac 12)}(nx^2)=\tfrac {1}{(2n)^n}\mathrm{He}_{2n}(\sqrt{2n}x),$$
where the last equality follows from~\eqref{eq:Hermite-Laguerre}. Note that in this case the rays through the roots of unity is just the real line, where the semicircle distribution is supported, see \eqref{eq:theo:hermite_polys_density}.

\vspace{2mm}
\noindent
{\sc Case III} in which $m=j+1\in\N$ and $c=b_1=\ldots=b_j=1$. Then $H=(\frac D {nm})^m$. Then, $\lsem \eee^{-\frac 1 {n^{m-1}m^m} D ^m}x^{nm}\rsem_{nm}$ and its weak limit are supported on $\{r\eee^{2\pi i/m}:r\ge 0\}$. After pushing forward to the negative half-line, the limit distribution $\mu$ is the same as in the following case.

\vspace{2mm}
\noindent
{\sc Case IV} in which $m=1$, $j\ge 1$ and $c=b_1=\ldots=b_j=1$. The operator $H$ is  $({xD})^jD/n^{j+1}$.
By Theorem \ref{thm:hamilton-flow}, we have
$$\lsem\eee^{-n^{-j+1} ({xD})^jD}x^n\rsem_n\toweak \mu_{+},$$
where $\mu_{+}(\cdot):=\mu(-(\cdot))$ and the measure $\mu$  on $(-\infty,0]$ is determined by the profile $\eee^{-g'(\alpha)}=\frac{\alpha^{j+1}}{1-\alpha}$, see Corollary~\ref{cor:real_zeros_hypergeo_polys_2}. By [Part I, Lemma 6.1], the $S$-transform $S_{\mu_{+}}$ of the limit $\mu_{+}$ is
\begin{equation}\label{eq:s_transform_via_profiles}
	S_{\mu_{+}}(z)=-\frac{1+z}{z}\eee^{g'(1+z)}=(1+z)^{-j}.
\end{equation}
According to Theorem 2.2 in~\cite{Arizmendi_Hasebe} this is an $S$-transform of a $\boxtimes$-infinitely divisible distribution.

\subsection{Appell polynomials and differential operators of Laguerre--P\'{o}lya class}

The differential operators of hypergeometric type are not the only ones which map monomials to polynomials with roots located on rays, such as the nonpositive real line, another class are those of Laguerre--P\'{o}lya type. Let us revisit a recent work by Andrew Campbell \cite{CampbellAppell}, where the limiting distribution of rescaled Appell polynomials have been identified to be freely infinitely divisible distributions. We won't provide a full introduction into the tools from free probability, and instead refer to \cite{jalowy_kabluchko_marynych_zeros_profiles_part_I}, and \cite{bercovici_voiculescu,CampbellAppell, nica_speicher_book,voiculescu_nica_dykema_book}.

Let $f$ be an entire function of the Laguerre--P\'{o}lya class, see \eqref{eq:polya_laguerre_def}, with $C=1, m=0$. Appell polynomials $A_n$ are real rooted, satisfy $A_n'=nA_{n-1}$
and can be represented as
\begin{align}\label{eq:A_n}
	A_n(x)=f(D)x^n,
\end{align}
where $D=\frac{d}{dx}$ is the differential operator. Recall that the $R$-transform of a distribution $\mu$ is defined by $R_\mu(t)=tG_\mu^{\leftarrow} (t)-1$, where $G_\mu^{\leftarrow}$ is the inverse of the Cauchy transform of $\mu$. Note that this definition differs from that of \cite{CampbellAppell} by a factor of $t$.
\begin{theorem}[{\cite[Theorem 2.1]{CampbellAppell}}]\label{theo:appell}
	Define $\hat A_n(x)= \big(f(\tfrac{D}{n})\big)^nx^n$. Suppose that $N<\infty$, then the sequence $\lsem\hat A_n\rsem$ converges weakly to an $\boxplus$-infinitely divisible distribution $\mu_A$ with the $R$-transform
	$$R_A(t)=\gamma t+\sigma^2 t^2 +
	\int\frac {t(t+x)}{1-tx}\frac{x^2}{x^2+1}d\nu(x),$$
	where $\nu=\sum_{j=1}^N\delta_{1/x_j}$ and $\gamma=-c-\sum_{j=1}^N\frac{1}{x_j^3+x_j}$.
\end{theorem}
\begin{proof}
	Let us write $f=f_1f_2f_3$, where $f_1(z)=\exp(z(c+\sum_j\frac 1 {x_j}))$, $f_2(z)=\exp(-\frac{\sigma^2}2z^2)$ and $f_3(z)=\prod_{j}\Big(1-\frac{z}{x_j}\Big)$, such that $\hat A_n(x)=\big(f_3(\tfrac{D}{n})\big)^n \big(f_2(\tfrac{D}{n})\big)^n\big(f_1(\tfrac{D}{n})\big)^nx^n$. First note that
	$$\lsem \big(f_2(\tfrac{D}{n})\big)^nx^n\rsem_n=\lsem\big(x+c+\sum_j\tfrac 1 {x_j}\big)^n \rsem_n=\delta_{-c-\sum_j\frac 1 {x_j}}.$$
	By Corollary \ref{cor:hermite_polys_zeros}, we have
	$$\lsem \big(f_1(\tfrac{D}{n})\big)^nx^n\rsem_n= \lsem \mathrm{He}_n( \sqrt {\tfrac n{\sigma^2}} x)\rsem_n\toweak \mathsf{sc_{\sigma^2}},$$
	where $\mathsf{sc}_{\sigma^2}=\mathsf{sc}_1(\tfrac{\cdot}{\sigma})$ is the semicircle distribution of variance $\sigma^2$. Furthermore, by Corollary~\ref{cor:laguerre_polys_zeros} and using the Rodrigues formula for Laguerre polynomials, we conclude that
	$$\lsem \big(1-\tfrac D {nx_j}\big)^nx^n\rsem_n=\lsem L_n^{(0)}(nx_jx)\rsem_n \toweak \mathsf{MP}_{x_j^{-1},1},\quad j=1,\ldots,N.$$
	Combining these, leads to the finite free convolution, see for instance Eq.~(2) in~\cite{marcus_spielman_srivastava}, that is,
	$$\hat A_n(x)=\big(x+c+\sum_j\tfrac 1 {x_j}\big)^n\boxplus_n  \mathrm{He}_n( \sqrt {\tfrac n{\sigma^2}} x)\boxplus_n L_n^{(0)}(nx_1x)\boxplus_n\dots\boxplus_n L_n^{(0)}(nx_Nx).$$
	Observe that by Theorem 1.3 in~\cite{marcus_spielman_srivastava} in combination with Theorems 6.32 and 6.31.3 in~\cite{Szego}, the roots of $\hat A_n(x)$ are upper bounded by an absolute constant. Thus, Theorem 3.5 in \cite{jalowy_kabluchko_marynych_zeros_profiles_part_I}, yields that the distribution $\lsem \hat A_n\rsem$ converges weakly to the free additive convolution $\delta_{-c-\sum_j\frac 1 {x_j}}\boxplus \mathsf{sc}_{\sigma^2}\boxplus \mathsf{MP}_{x_1^{-1},1}\boxplus\dots\boxplus \mathsf{MP}_{x_N^{-1},1}$. Its $R$ transform is a sum of the individual $R$-transforms $R_{\delta_x}(t)=xt$, $R_{\mathsf{sc}_{\sigma^2}}(t)=\sigma^2 t^2$, and $R_{\mathsf{MP}_{x_j^{-1},1}}(t)=\frac {t}{x_j-t}$. Summing up and rearranging yields the claim.
\end{proof}

\section{Characteristic polynomial of the sample covariance matrix}

The goal of this section is to demonstrate how the converse Theorem~\ref{theo:zeros_imply_exp_profile} can be used to derive the asymptotic behavior of the coefficients. We focus on random polynomials arising from a well-known ensemble of random matrices. Fortunately, the empirical zero distributions of these polynomials converge almost surely, which allows us to apply Theorem~\ref{theo:zeros_imply_exp_profile} pathwise.

Let $(X_{i,j})_{i,j\geq 1}$ be a double array of independent identically distributed real random variables satisfying
$$
\mathbb{E}[X_{i,j}]=0,\quad {\rm Var}\,[X_{i,j}]=:\sigma^2\in (0,+\infty).
$$
For $m,n\in\mathbb{N}$, put ${\bf X}:=[X_{i,j}]_{1\leq i\leq n,1\leq j\leq m}$. The symmetric random matrix
\begin{equation}\label{eq:sample_covariance_matrix_def}
{\bf M}(m,n):=m^{-1}{\bf X}{\bf X}^{T}
\end{equation}
is called the sample covariance matrix. Let $(\xi_{j:n})_{1\leq j\leq n}\subset [0,\infty)$ be the set of its eigenvalues and $D_{m,n}(x):=\det({\bf M}(m,n)-x{\bf Id}_n)=\prod_{j=1}^{n}(\xi_{j:n}-x)$ be its characteristic polynomial. Assume that $m=m(n)$ is such that
\begin{equation}\label{eq:c_n_wishart}
\lim_{n\to\infty}\frac{n}{m(n)}=\lambda\in (0,\infty).
\end{equation}
Put $\lambda^{*}:=\max(1-\lambda^{-1},0)$. According to~\cite[Theorem 3.6]{BaiSilverstein} and assuming~\eqref{eq:c_n_wishart}, with probability one
\begin{equation}\label{eq:convergence_to_MP}
\frac{1}{n}\sum_{j=1}^{n}\delta_{-\xi_{j:n}}=\lsem D_{m,n}(-x)\rsem_n\toweak {\sf MP}_{\sigma^2,\lambda}(-(\cdot)),
\end{equation}
where ${\sf MP}_{\sigma^2,\lambda}$ is the Marchenko--Pastur distribution with the density~\eqref{eq:MP_def} if $\lambda\le 1$, and is a sum of an atom $\lambda^*\delta_0$ at zero and the distribution with the same density~\eqref{eq:MP_def} of the total mass $\lambda^{-1}$ if $\lambda>1$.

We shall now apply Theorem~\ref{theo:zeros_imply_exp_profile} to deduce a strong law of large numbers for the coefficients $a_{k:n}$ of the normalized polynomial $D_{m,n}(-x)$. Observe that the limiting measure $\mu$ satisfies $\underline{m}=\mu(\{0\})=\max(1-\lambda^{-1},0)$ and $\overline{m}=1-\mu(\{-\infty\})=1$.

\begin{theorem}\label{thm:char_poly_sample_covariance_matrix}
Recall $\lim_{n\to\infty}\frac{n}{m(n)}=\lambda\in (0,\infty)$ and $\lambda^{*}=\max(1-\lambda^{-1},0)$. Define the coefficients of the normalized characteristic polynomial of covariance matrices as
$$
a_{k:n}:=[x^k]\frac{\det({\bf M}(m,n)+x{\bf Id}_n)}{\det({\bf M}(m,n)+{\bf Id}_n)},\quad k\in \{0,1,\ldots,n\}.
$$
Then, for all sufficiently small $\eps > 0$, with probability one,
\begin{equation}\label{eq:char_poly_sample_covariance_matrix_claim1}
\sup_{(\lambda^{*}+\eps) n \leq k  \leq (1-\eps)n}
\left|\frac 1n  \log a_{k:n}-g_{\bf M}\left(\frac k  n\right)  \right| \ton 0,
\end{equation}
as well as
\begin{equation}\label{char_poly_sample_covariance_matrix_claim2}
\sup_{0\leq k  \leq (\lambda^{*}-\eps)n}  \frac 1n  \log a_{k:n} \ton  -\infty.
\end{equation}
Here $g_{\bf M}:(\lambda_{*},1)\to (-\infty,0]$ is a strictly concave infinitely differentiable function, defined by
formulas~\eqref{eq:g_SCM_1}-\eqref{eq:g_SCM_3} below and satisfying $\eee^{-g_{\bf M}^{\prime}(\alpha)} =\frac{\sigma^2\alpha(\lambda\alpha+1-\lambda)}{1-\alpha}$, $\lambda^*<\alpha<1$.
\end{theorem}

Naturally, for $\sigma^2=\lambda^{-1}$, this coincides with the exponential profile of Laguerre polynomials of Corollary~\ref{cor:laguerre_polys_zeros} corresponding to $\gamma^{\ast}=\lambda^{-1}-1$.

\begin{proof}
We only need to compute $g_{\bf M}$. According to~\cite[Lemma 3.11]{BaiSilverstein}, see also Eq.~\eqref{eq:marchenko_pastur_Stieltjes}, the Cauchy transform of the right-hand side of~\eqref{eq:convergence_to_MP} is equal to\footnote{Observe that~\eqref{eq:marchenko_pastur_Stieltjes} is stated for $\lambda\in (0,1]$ only. If $\lambda>1$, then the right-hand side of~\eqref{eq:marchenko_pastur_Stieltjes} is the Cauchy transform of $\max(1-{\lambda}^{-1},0)\delta_0+{\sf MP}_{\sigma^2,\lambda}$, see \cite[Lemma 3.11]{BaiSilverstein} and their definition of the Marchenko--Pastur law in Eq.~(3.1.1).}
\begin{equation}\label{eq:marchenko_pastur_Stieltjes_inverted}
G(t)=-\frac {t+\sigma^2(1-\lambda)+\sqrt{(\lambda_{+}+t)(\lambda_{-}+t)}}{2\sigma^2\lambda t},\quad t\in  \C \bsl [-\lambda_+,-\lambda_-],
\end{equation}
where the branch is chosen such that $t\mapsto \sqrt{(\lambda_{+}+t)(\lambda_{-}+t)}\sim -t$, as $|t|\to+\infty$. One can check that $t\mapsto tG(t)$ is a solution to the quadratic equation
$$
\sigma^2 \lambda \alpha^2+(t+\sigma^2(1-\lambda))\alpha-t=0.
$$
Solving this equation for $t$ we find that the inverse of $t\mapsto tG(t)$ is
$$
\eee^{-g_{\bf M}^{\prime}(\alpha)}=\frac{\sigma^2\lambda \alpha^2+\sigma^2(1-\lambda)\alpha}{1-\alpha}=-\sigma^2\lambda \alpha+\frac{\sigma^2\alpha}{1-\alpha}=\frac{\sigma^2\alpha(\lambda\alpha+1-\lambda)}{1-\alpha}.
$$
Observe that this function if positive and nondecreasing for $\alpha\in (\lambda^{\ast},1)$. Taking the logarithms and integrating we obtain,
\begin{equation}\label{eq:g_SCM_1}
g_{\bf M}(\alpha)=-2\alpha\log\sigma+\left(-\alpha+1-\frac{1}{\lambda}\right)\log (\lambda\alpha+1-\lambda)+\alpha-\alpha\log\alpha-(1-\alpha)\log(1-\alpha)+C.
\end{equation}
The additive constant $C$ can be determined from the relation
$$
\sup_{\alpha\in (\lambda^{\ast},1)}g_{\bf M}(\alpha)=0,
$$
see formula~\eqref{eq:main_polys_converse_assumption_p_n(1)}. After elementary but lengthy calculations this yields
\begin{equation}\label{eq:g_SCM_2}
C=\frac{1}{\lambda}\log(1-\alpha^{\ast})-\frac{1}{\lambda}\log\alpha^{\ast}+\log \alpha^{\ast}-\alpha^{\ast}-2\left(\frac{1}{\lambda}-1\right)\log \sigma,
\end{equation}
where $\alpha^{\ast}$ is the unique solution on $(\lambda^{\ast},1)$ to the quadratic equation
\begin{equation}\label{eq:g_SCM_3}
\frac{1-\alpha^{\ast}}{\alpha^{\ast}(\lambda\alpha^{\ast}+1-\lambda)}=\sigma^2.
\end{equation}
The proof is complete.
\end{proof}

\appendix

\section{Lambert function}\label{sec:lambert_properties}
Here we collect some relevant properties of the Lambert function. We refer to~\cite{corless_etal} and~\cite{mezo_book_lambert_function} for more details.
The Lambert function is a  multivalued analytic function defined by the implicit equation
$$
W(z) \eee^{W(z)} = z, \qquad z\in\C.
$$
It has a branch point at $z=-1/\eee$ and infinitely many branches. We need only the principal branch $W_0$ which is defined on the whole complex plane with a branch cut at $(-\infty, -1/\eee]$. On the real line, the function $w\mapsto w\eee^w=z$ has a unique minimum $-1/\eee$ attained at $w=-1$. Thus, there is a well-defined inverse function $W_0:(-1/\eee , \infty) \to (-1,\infty)$, called the principal branch, which is monotone increasing, and satisfies
$$
W_0(z) \eee^{W_0(z)} = z, \qquad \lim_{x\downarrow -1/\eee} W_0(x) = -1, \qquad W_0(0)=0, \qquad  \lim_{x\uparrow +\infty} W_0(x) = +\infty.
$$
The principal branch $W_0$ can be analytically continued to the slitted complex plane $\C\backslash (-\infty,-1/\eee]$. In Theorem~\ref{theo:touchard_polys_zeros}, we need the limit values of $W_0(z)$ when the complex variable $z$ approaches the branch cut $(-\infty, -1/\eee]$. Such limit value depends on whether $z$ stays in the upper half-plane or in the lower half-plane. We can define
$$
W_0(x + \ii 0) :=  \lim_{\eps\downarrow 0} W_0(x + \ii \eps),
\qquad
W_0(x - \ii 0) :=  \lim_{\eps\downarrow 0} W_0(x - \ii \eps) = \overline {W_0(x + \ii 0)},
\qquad
x\in (-\infty, -1/\eee).
$$
It is also known that $\Im W_0(x + \ii 0)$ is a decreasing function of  $x\in (-\infty, -1/\eee)$ and that
$$
\lim_{x\downarrow -\infty} \Im W_0(x + \ii 0) = \pi,
\qquad
\lim_{x\uparrow -1/\eee} \Im W_0(x + \ii 0) = 0.
$$
We  need the following formula for the derivative of the Lambert function:
\begin{equation}\label{eq:W_0_der}
W_0'(z) = \frac{W_0(z)}{z(1+W_0(z))}, \qquad z \in \C\backslash (-\infty,-1/\eee], \quad z\neq 0.
\end{equation}
It is well known that for every $v\in \C$ and for all complex $|z|<1/\eee$ the following expansion holds:
\begin{equation}\label{eq:taylor_W_0_general}
\left(\frac{W_0(z)}{z}\right)^{v}
=
1 + \sum_{n=1}^\infty (-1)^{n} \frac{v(n+v)^{n-1}}{n!} z^n
\end{equation}
This identity easily follows from the Lagrange inversion formula in B\"urmann's form~\cite[pp.~732--733]{flajolet_sedgewick_book}. This formula states that if the formal power series $y(z)$ is defined as the solution to $y/\phi(y) = z$, where $\phi(u) = \sum_{k=0}^\infty \phi_k u^k$ is a formal power series with $\phi_0\neq 0$, then
$$
[z^n] \left(\frac {y(z)}{z}\right)^v = \frac {v}{n+v} [u^n] \phi(u)^{n+v}.
$$
Taking $\phi(u) = \eee^{-u}$, one easily arrives at~\eqref{eq:taylor_W_0_general}.

Taking $v=-1$ in~\eqref{eq:taylor_W_0_general},
and multiplying by $z$, we obtain the formula
\begin{equation}\label{eq:taylor_exp_W_0}
\eee^{W_0(z)} = \frac{z}{W_0(z)} = 1 + \sum_{n=1}^\infty (-1)^{n+1}\frac{ (n-1)^{n-1}}{n!} z^n
=
1 + \sum_{k=0}^\infty\frac{ (-k)^{k}}{(k+1)!} z^{k+1}.
\end{equation}

\section*{Acknowledgements}
The authors thank Sung-Soo Byun for suggesting a q-orthogonal polynomial. ZK was supported by the German Research Foundation under Germany's Excellence Strategy EXC 2044 - 390685587, Mathematics M\"{u}nster: Dynamics - Geometry - Structure. ZK and JJ have been supported by the DFG priority program SPP 2265 Random Geometric Systems.


\end{document}